\documentclass{siamltex}
\usepackage[T1]{fontenc}
\usepackage[utf8]{inputenc}
\usepackage{lmodern}
\usepackage{amssymb}
\usepackage{amsmath}
\usepackage{amsfonts}
\usepackage{sw20siam}
\usepackage{graphicx}
\usepackage{epsfig}
\usepackage{multicol}
\usepackage{epstopdf}
\usepackage{float}
\usepackage{subfig}
\providecommand{\U}[1]{\protect\rule{.1in}{.1in}}
\begin{document}

\title{On URANS Congruity with Time Averaging:\\Analytical laws suggest improved models }
\author{William Layton \thanks{%
Department of Mathematics, University of Pittsburgh, Pittsburgh, PA 15260,
USA, wjl@pitt.edu; The research herein was partially supported by NSF\
grant DMS 1817542.} \and Michael McLaughlin 
\thanks{%
Department of Mathematics, University of Pittsburgh, Pittsburgh, PA 15260,
USA, mem266@pitt.edu; The research herein was partially supported by NSF grant DMS 1817542.}}
\date{\today}
\maketitle


\begin{abstract}
The standard $1-$equation model \ of turbulence was first derived by Prandtl
and has evolved to be a common method for practical flow simulations. Five
fundamental laws that any URANS model should satisfy are%
\[%
\begin{array}
[c]{ccc}%
\text{\textbf{1.}} & \text{Time window:} &
\begin{array}
[c]{c}%
\tau\downarrow0\text{ implies }v_{\text{{\small URANS}}}\rightarrow
u_{\text{{\small NSE}}}\text{ \&}\\
\text{ }\tau\uparrow\text{implies }\nu_{T}\uparrow
\end{array}
\\
\text{\textbf{2.}} & \text{ \ }l(x)=0\ \text{at walls:} & l(x)\rightarrow
0\text{ as }x\rightarrow walls,\\
\text{\textbf{3.}} & \text{ Bounded energy:} & \sup_{t}\int\frac{1}%
{2}|v(x,t)|^{2}+k(x,t)dx<\infty\\
\text{\textbf{4.}} &
\begin{array}
[c]{c}%
\text{Statistical }\\
\text{equilibrium:}%
\end{array}
& \lim\sup_{T\rightarrow\infty}\frac{1}{T}\int_{0}^{T}\varepsilon
_{\text{model}}(t)dt=\mathcal{O}\left(  \frac{U^{3}}{L}\right)  \\
\text{\textbf{5.}} &
\begin{array}
[c]{c}%
\text{Backscatter}\\
\text{possible:}%
\end{array}
& \text{(without negative viscosities)}%
\end{array}
\]
This report proves that a \textit{kinematic} specification of the model's
turbulence lengthscale by
\[
l(x,t)=\sqrt{2}k^{1/2}(x,t)\tau\text{ },
\]
where $\tau$\ is the time filter window, results in a $1-$equation model
satisfying Conditions 1,2,3,4 without model tweaks, adjustments or wall
damping multipliers.

\end{abstract}

\section{Introduction}

URANS (\textit{unsteady Reynolds averaged Navier-Stokes}) models of turbulence
are derived\footnote{URANS models are also constructed ad hoc simply by adding
$\frac{\partial v}{\partial t}$\ to a RANS model without regard to where the
term originates. Formulation via averaging over a finite time window is a
coherent source for the term.} commonly to produce a velocity, $v(x,t)\simeq
\overline{u}(x,t)$, that approximates a finite time window average of the
Navier-Stokes velocity $u(x,t)$%
\begin{equation}
\overline{u}(x,t)=\frac{1}{\tau}\int_{t-\tau}^{t}u(x,t^{\prime})dt^{\prime
}.\label{eqTimeMean}%
\end{equation}
From this connection flows 5 fundamental conditions (below) that a coherent
URANS\ model should satisfy and that few do. Herein we delineate these
conditions and show that, for the standard $1-$equation model, a new kinematic
turbulence length scale results in a simpler model satisfying 4 of the 5.

The first condition is a simple observation that the time window $\tau$ should
influence the model, as $\tau\rightarrow0$ the model should revert to the NSE
(Navier-Stokes equations) and as $\tau$\ increases, more time scales are
filtered and thus the eddy viscosity should increase.

\textbf{Condition 1:} \textit{The filter window }$\tau$\textit{\ should appear
as a model parameter. As }$\tau\rightarrow 0$\textit{\ the model reverts to the
NSE. As }$\tau$\textit{\ increases, the model eddy viscosity }$\nu_{T}(\cdot
)$\textit{\ increases.}

We consider herein $1-$equation models of turbulence. These have deficiencies
but nevertheless include models considered to have good predictive accuracy
and low cost, e.g., Spalart \cite{S15} and Figure 2 p.8 in Xiao and Cinnella
\cite{XC18}. The standard $1-$equation model (from which all have evolved),
introduced by Prandtl \cite{P45}, is%
\begin{gather}
v_{t}+v\cdot\nabla v-\nabla\cdot\left(  \left[  2\nu+\mu l\sqrt{k}\right]
\nabla^{s}v\right)  +\nabla p=f(x)\text{, }\nonumber\\
\nabla\cdot v=0\text{, }\label{eq:1EqnModel}\\
k_{t}+v\cdot\nabla k-\nabla\cdot\left(  \left[  \nu+\mu l\sqrt{k}\right]
\nabla k\right)  +\frac{1}{l}k\sqrt{k}=\mu l\sqrt{k}|\nabla^{s}v|^{2}%
\text{.}\nonumber
\end{gather}
Briefly, $p(x,t)$ is a pressure, $f(x)$ is a smooth, divergence free
($\nabla\cdot f=0$) body force, $\mu\simeq0.55$ is a calibration
parameter\footnote{Pope \cite{Pope} calculates the value $\mu=0.55$ from the
($3d$) law of the wall. An analogy with the kinetic theory of gasses (for
which $\nu_{T}=\frac{1}{3}lU$) yields the value $\mu=\frac{1}{3}\sqrt{2/d}$
which gives $\mu\simeq0.33$\ in 2d and $\mu\simeq0.27$\ in $3d$, Davidson
\cite{D15} p. 114, eqn. (4.11a).}, $\nabla^{s}v=(\nabla v+\nabla^{T}v)/2$ is
the deformation tensor, and $k(x,t)$ is the model approximation to the
fluctuations' kinetic energy distribution, $\frac{1}{2}|(u-\overline
{u})(x,t)|^{2}$. The eddy viscosity coefficient
\[
\nu_{T}(\cdot)=\mu l\sqrt{k}%
\]
(the Prandtl-Kolmogorov formula) is a dimensionally consistent expression of
the observed increase of mixing with turbulence and of the physical idea of
Saint-Venant \cite{S43} that this mixing increases with "\textit{the intensity
of the whirling agitation}", \cite{Darrigol}, p.235. The $k-$equation
describes the turbulent kinetic energy evolution; see \cite{CL} p.99, Section
4.4, \cite{D15}, \cite{MP} p.60, Section 5.3 or \cite{Pope} p.369, Section
10.3,\ for a derivation. The model (\ref{eq:1EqnModel}) holds in a flow domain
$\Omega$\ with initial conditions, $v(x,0)$ and $k(x,0)$, and (here
$L-$periodic or no-slip) $v,k$ boundary conditions on the boundary
$\partial\Omega$.

The parameter of interest herein is the turbulence length-scale $l=l(x)$,
first postulated by Taylor in 1915 \cite{T15}. It varies from model to model,
flow subregion to subregion (requiring fore knowledge of their locations,
\cite{S15}) and must be specified by the user; see \cite{Wilcox} for many
examples of how $l(x)$\ is chosen in various subregions. The simplest case is
channel flow for which%
\[
l_{0}(x)=\min\{0.41y,0.082\mathcal{R}e^{-1/2}\}
\]
where $y$\ is the wall normal distance, Wilcox \cite{Wilcox} Ch. 3, eqn.
(3.99) p.76.

Model solutions are approximations to averages of velocities of the
incompressible Navier-Stokes equations. Other fundamental physical properties
of NSE solutions (inherited by averages) should also be preserved by the
model. These properties include:

\textbf{Condition 2:} \textit{The turbulence length-scale }$l(x)$%
\textit{\ must }$l(x)\rightarrow0$\textit{\ as }$x\rightarrow walls$\textit{.}

Condition 2 follows since the eddy viscosity term approximates the Reynolds
stresses and%
\[
\mu l\sqrt{k}\nabla^{s}v\simeq u^{\prime}u^{\prime}\text{ which }%
\rightarrow0\text{ at walls like }\mathcal{O}(\text{\textit{wall-distance}%
}^{2}).
\]
Specifications of $l(x)$ violating this are often observed to over-dissipate
solutions (in many tests and now with mathematical support \cite{P17}).

\textbf{Condition 3:}\textit{\ (Finite kinetic energy) The model's
representation of the total kinetic energy in the fluid must be uniformly
bounded in time:}%
\[
\int_{\Omega}\frac{1}{2}|v(x,t)|^{2}+k(x,t)dx\leq Const.<\infty\text{
uniformly in time.}%
\]
The kinetic energy (per unit volume)$\frac{1}{|\Omega|}\int\frac{1}{2}%
|u|^{2}dx$, is distributed between means and fluctuations in the model as%
\[
\frac{1}{|\Omega|}\int_{\Omega}\frac{1}{2}|v(x,t)|^{2}+k(x,t)dx\simeq\frac
{1}{|\Omega|}\int_{\Omega}\frac{1}{2}|u(x,t)|^{2}dx<\infty.
\]
This property for the NSE represents the physical fact that bounded energy
input does not grow to unbounded energy solutions.

\textbf{Condition 4:}\textit{\ (Time-averaged statistical equilibrium) The
time average of the model's total energy dissipation rate, }$\varepsilon
_{\text{model}}$\textit{\ (\ref{eq:ModelEnergyDissipationRate}) below, should
be at most the time average energy input rate:}%
\[
\lim\sup_{T\rightarrow\infty}\frac{1}{T}\int_{0}^{T}\varepsilon_{\text{model}%
}(t)dt\leq Const.\frac{U^{3}}{L}\text{, uniformly in }\mathcal{R}e\text{.}%
\]
The most common failure model for turbulence models is over-dissipation.
Condition 4 expresses aggregate non-over-dissipatiopn. The energy dissipation
rate is a fundamental statistic of turbulence, e.g., \cite{Pope}, \cite{V15}.
This balance is observed in physical experiments \cite{Frisch}, \cite{V15} and
has been proven for the NSE, \cite{DC92}, \cite{DF02}, \cite{DG}.

The fifth condition is that the model allows an intermittent flow of energy
from fluctuations back to means. This energy flow is important, e.g.
\cite{Starr}, \cite{VGF94}, less well understood and not addressed herein; for
background see \cite{JL16}.

\textbf{Condition 5:}\textit{\ The model allows flow of energy from
fluctuations back to means without negative eddy viscosities. This energy flow
has space time average zero. }

To develop Conditions 3 and 4, multiple the $v-$equation (\ref{eq:1EqnModel})
by $v$ and integrate over $\Omega$. Add to this the $k-$ equation integrated
over $\Omega$. After standard manipulations and cancellations of terms there
follows the model's global energy balance%
\begin{gather}
\frac{d}{dt}\int_{\Omega}\frac{1}{2}|v(x,t)|^{2}+k(x,t)dx+\int_{\Omega}%
2\nu|\nabla^{s}v(x,t)|^{2}+\frac{1}{l(x)}k^{3/2}%
(x,t)dx\label{eq:StandardMoelEnergyBalance}\\
=\int_{\Omega}f(x)\cdot v(x,t)dx.\nonumber
\end{gather}
Thus, for the $1-$equation model we have (per unit volume)%
\begin{align}
\text{Kinetic energy }  &  \text{=}\text{ }\frac{1}{|\Omega|}\int_{\Omega
}\frac{1}{2}|v(x,t)|^{2}+k(x,t)dx,\nonumber\\
\text{Dissipation rate }\varepsilon_{\text{model}}(t)  &  =\frac{1}{|\Omega
|}\int_{\Omega}2\nu|\nabla^{s}v(x,t)|^{2}+\frac{1}{l(x)}k^{3/2}(x,t)dx
\label{eq:ModelEnergyDissipationRate}%
\end{align}

\textbf{The standard }$\mathbf{1-}$\textbf{equation model has difficulties
with all 5 conditions.} Conditions 1 and 5 are clearly violated. The second,
$l(x)\rightarrow0$\textit{\ at walls}, is not easily enforced for complex
boundaries; it is further complicated in current models, e.g., Spalart
\cite{S15}, Wilcox \cite{Wilcox}, by requiring user input of (unknown)
subregion locations where different formulas for $l(x)$\ are used. Conditions
3 and 4 also seem to be unknown for the standard model; they do not follow
from standard differential inequalities due to the mismatch of the powers of
$k$ in the energy term and the dissipation term.

\textbf{The correction herein is a kinematic }$l(x,t)$\textbf{.} We prove
herein that \textit{a kinematic\footnote{This can also be argued to be a
\textit{dynamic} choice since the estimate of $|u^{\prime}|$ \ in $l(x,t)$\ is
calculated from an (approximate) causal law. } turbulence length-scale
enforces Condition 1,2,3 and 4} as well as simplifying the model. In its
origin, the turbulence length-scale (then called a \textit{mixing length}) was
an analog to the mean free pass in the kinetic theory of gases. It represented
the distance two fluctuating structures must traverse to interact. Prandtl
\cite{P26} in 1926 also mentioned a second possibility:

\begin{center}
\textit{... the distance traversed by a mass of this type before it becomes
blended in with neighboring masses...}.
\end{center}

The idea expressed above is ambiguous but can be interpreted as suggesting
$l=|u^{\prime}(x,t)|\tau$, i.e., the \textit{distance a fluctuating eddy
travels in one time unit.} This choice means to select a turbulence time scale
$\tau$ (e.g., from (\ref{eqTimeMean})) and, as $|u^{\prime}|\simeq\sqrt
{2}k(x,t)^{1/2}$, define\footnote{The $k-$equation and a weak maximum
principle imply $k(x,t)\geq0$, following \cite{WF02}, \cite{LM93}. Thus,
$k^{1/2}$ is well defined.} $l(x,t) $ \textit{kinematically} by
\begin{equation}
l(x,t)=\sqrt{2}k(x,t)^{1/2}\tau. \label{eq:DynamicMixingLength}%
\end{equation}
With this choice the \textit{time window }$\tau$\textit{\ enters into the
model}. To our knowledge, (\ref{eq:DynamicMixingLength}) is little developed.
Recently in \cite{JL14b} the idea of $l=|u^{\prime}|\tau$ has been shown to
have positive features in ensemble simulations. With
(\ref{eq:DynamicMixingLength}), the model (\ref{eq:1EqnModel}) is modified to%
\begin{gather}
v_{t}+v\cdot\nabla v-\nabla\cdot\left(  \left[  2\nu+\sqrt{2}\mu k\tau\right]
\nabla^{s}v\right)  +\nabla p=f(x)\text{, }\nonumber\\
\nabla\cdot v=0\text{, }\label{eq:New1EqnModel}\\
k_{t}+v\cdot\nabla k-\nabla\cdot\left(  \left[  \nu+\sqrt{2}\mu k\tau\right]
\nabla k\right)  +\frac{\sqrt{2}}{2}\tau^{-1}k=\sqrt{2}\mu k\tau|\nabla
^{s}v|^{2}\text{.}\nonumber
\end{gather}

Let $L,U$ denote large length and velocity scales, defined precisely in
Section 2, equation (\ref{eq:FLUdefinition}), $\mathcal{R}e=LU/\nu$ the usual
Reynolds number and let $T^{\ast}=L/U$ denote the large scale turnover time.
The main result herein is that with the kinematic length scale selection
(\ref{eq:DynamicMixingLength}) conditions 1-4 are now satisfied.

\begin{theorem}
Let $\mu,\tau$\ be positive and $\Omega$ a bounded regular domain. Let
\[
l(x,t)=\sqrt{2}k(x,t)^{1/2}\tau.
\]
Then, condition 1 holds.

Suppose the boundary conditions are no-slip ($v=0,k=0$ on $\partial\Omega$).
Then, Condition 2 is satisfied. At walls
\[
l(x)\rightarrow0\mathit{\ }\text{\textit{as}}\mathit{\ }x\rightarrow walls.
\]

Suppose the model's energy inequality, equation (\ref{eq:EnergyINeq}) below,
holds. If the boundary conditions are either no slip or periodic with zero
mean for $v$ and periodic for $k$, (\ref{eq:PeriodicZeroMean}) below,
\textit{Condition 3 also holds}:%
\[
\int_{\Omega}\frac{1}{2}|v(x,t)|^{2}+k(x,t)dx\leq Const.<\infty\text{
uniformly in time.}%
\]
The model's energy dissipation rate is%
\[
\varepsilon_{\text{model}}(t)=\frac{1}{|\Omega|}\int_{\Omega}2\nu|\nabla
^{s}v(x,t)|^{2}+\frac{\sqrt{2}}{2}\tau^{-1}k(x,t)dx.
\]
\textit{Time averages of the model's energy dissipation rate are finite}:%
\[
\lim\sup_{T\rightarrow\infty}\frac{1}{T}\int_{0}^{T}\varepsilon_{\text{model}%
}(t)dt<\infty.
\]

Suppose the boundary conditions are either periodic with zero mean for $v$ and
periodic for $k$, (\ref{eq:PeriodicZeroMean}) below, or no-slip ($v=0,k=0$ on
the boundary) and the body force satisfies $f(x)=0$ on the boundary. If the
selected time averaging window satisfies%
\[
\frac{\tau}{T^{\ast}}\leq\frac{1}{\sqrt{\mu}}\text{ \ }\left(  \simeq
1.35\text{ for }\mu=0.55\right)
\]
then \textit{Condition 4 holds} uniformly in the Reynolds number
\[
\lim\sup_{T\rightarrow\infty}\frac{1}{T}\int_{0}^{T}\varepsilon_{\text{model}%
}(t)dt\leq4\left(  1+\mathcal{R}e^{-1}\right)  \frac{U^{3}}{L}\text{.}%
\]

\end{theorem}

\begin{proof}
The proof that Condition 4 holds will be presented in Section 3. The reminder
is proven as follows. Condition 1 is obvious. Since $l(x,t)=\sqrt
{2}k(x,t)^{1/2}\tau$\ and $k(x,t)$ vanishes at walls it follows that so does
$l(x,t)$ so Condition 2 holds.

In the energy inequality (\ref{eq:EnergyINeq}), $l(x,t)=\sqrt{2}%
k(x,t)^{1/2}\tau$ yields
\begin{gather}
\frac{d}{dt}\int_{\Omega}\frac{1}{2}|v(x,t)|^{2}+k(x,t)dx+\int_{\Omega}%
2\nu|\nabla^{s}v(x,t)|^{2}+\frac{\sqrt{2}}{2}\tau^{-1}k(x,t)dx\nonumber\\
\leq\int_{\Omega}f(x)\cdot v(x,t)dx\text{.} \label{eq:EnergyBalance2}%
\end{gather}
By Korn's inequality and the Poincar\'{e}-Friedrichs inequality
\begin{align*}
\alpha\int_{\Omega}\frac{1}{2}|v(x,t)|^{2}+k(x,t)dx  &  \leq\int_{\Omega}%
2\nu|\nabla^{s}v(x,t)|^{2}+\frac{\sqrt{2}}{2}\tau^{-1}k(x,t)dx,\\
\text{where }\alpha &  =\alpha(C_{PF},\nu,\tau)>0.
\end{align*}
Let $y(t)=\int\frac{1}{2}|v(x,t)|^{2}+k(x,t)dx$. Thus, $y(t)$ satisfies%
\[
y^{\prime}(t)+\alpha y(t)\leq\int_{\Omega}f(x)\cdot v(x,t)dx\leq\frac{\alpha
}{2}y(t)+C(\alpha)\int_{\Omega}|f|^{2}dx.
\]
An integrating factor then implies%
\[
y(t)\leq e^{-\frac{\alpha}{2}t}y(0)+\left(  C(\alpha)\int_{\Omega}%
|f|^{2}dx\right)  \int_{0}^{t}e^{-\frac{\alpha}{2}(t-s)}ds
\]
which is uniformly bounded in time, verifying Condition 3.

For the last claim, time average the energy balance (\ref{eq:EnergyBalance2}).
The result can be compressed to read%
\[
\frac{y(T)-y(0)}{T}+\frac{1}{T}\int_{0}^{T}\varepsilon_{\text{model}%
}(t)dt=\frac{1}{T}\int_{0}^{T}\left(  \int_{\Omega}f(x)\cdot v(x,t)dx\right)
dt
\]
The first term on the left hand side is $\mathcal{O}(\frac{1}{T})$ since
$y(t)$ is uniformly bounded. The RHS is also uniformly in $T$ bounded (again
since $y(t)$ is uniformly bounded). Thus so is $\frac{1}{T}\int_{0}%
^{T}\varepsilon_{\text{model}}(t)dt$.
\end{proof}

The estimate $\varepsilon\simeq$\ $U^{3}/L$ in Theorem 1 is consistent as
$Re\rightarrow\infty$ with both phenomenology, \cite{Pope}, and the rate
proven for the Navier-Stokes equations in \cite{Wang97}, \cite{DF02},
\cite{DC92}. Building on this work, the proof in Section consists of
estimating 4 key terms. The first 3 are a close parallel to the NSE analysis
in these papers and the fourth is model specific.

The main contribution herein is then recognition that several flaws of the
model (\ref{eq:1EqnModel}) originate in the turbulence length-scale
specification. These are corrected by the kinematic choice
(\ref{eq:DynamicMixingLength}) rather than by calibrating $l$ with increased
complexity. The second main contribution is the proof in Section 3 that the
kinematic choice does not over dissipate, i.e., Condition 4 holds.

\textbf{Model existence is an open problem}. The proof of Theorem 1 requires
assuming weak solutions of the model exist and satisfy an energy inequality
(i.e., (\ref{eq:StandardMoelEnergyBalance}) with $=$\ replaced by $\leq$),
$k(x,t)\geq0$ and that in the model's weak formulation the test function may
be chosen to be the (smooth) body force $f(x)$. Such a theory for the standard
model (with static $l=l(x)$) has been developed over 20+ years of difficult
progress from intense effort including \cite{L97}, with positivity of $k$
established in \cite{LM93}, see also \cite{WF02}, existence of suitable weak
solutions in \cite{BLM11}, culminating in Chapter 8 of \cite{CL} and
\cite{BM18} including an energy inequality (with equality an open problem) and
uniqueness under restrictive conditions. Conditions 3 and 4 are open problems
for the standard model. Based on this work we conjecture that an existence
theory, while not the topic of this report, may be possible for the (related)
$1-$equation model with kinematic length scale (\ref{eq:New1EqnModel}).

\section{Preliminaries and notation}

This section will develop Condition 4, that after time averaging
$\varepsilon_{\text{model}}\simeq$\ $U^{3}/L$, and present notation and
preliminaries needed for the proof in Section 3. We impose periodic boundary
conditions on $k(x,t)$ and periodic with zero mean boundary conditions on
$v,p,v_{0},f$. Periodicity and zero mean denote respectively%
\begin{equation}
\text{\textit{Periodic}: }\phi(x+L_{\Omega}e_{j},t)=\phi(x,t)\text{ and
\textit{Zero mean}:}\int_{\Omega}\phi dx=0\,. \label{eq:PeriodicZeroMean}%
\end{equation}
The proof when the boundary conditions are no-slip, $v=0,k=0$ on
$\partial\Omega$, and $f(x)=0$ on $\partial\Omega$\ will be omitted. It\ is
exactly the same as in the periodic case.

\textbf{Notation used in the proof.\ }The long time average of a function
$\phi(t)$\ is%
\begin{gather*}
\left\langle \phi\right\rangle =\lim\sup_{T\rightarrow\infty}\frac{1}{T}%
\int_{0}^{T}\phi(t)dt\text{ }\text{and satisfies}\\
\left\langle \phi\psi\right\rangle \leq\left\langle |\phi|^{2}\right\rangle
^{1/2}\left\langle |\psi|^{2}\right\rangle ^{1/2}\text{ and }\left\langle
\left\langle \phi\right\rangle \right\rangle =\left\langle \phi\right\rangle .
\end{gather*}
The usual $L^{2}(\Omega)$ norm, inner product and $L^{p}(\Omega)$ norm are
$||\cdot||,(\cdot,\cdot),||\cdot||_{p}$.

\textbf{Preliminaries.} Define the global velocity scale\footnote{It will
simplify the proofs not to scale also by the number of components. This can
easily be done in the final result.} $U$, the body force scale $F$ and large
length scale $L$ by%
\begin{equation}
\left.
\begin{array}
[c]{c}%
F=\left(  \frac{1}{|\Omega|}\int_{\Omega}|f(x)|^{2}dx\right)  ^{1/2}\text{,
}\\
L=\min\left[  L_{\Omega},\frac{F}{\sup_{x\in\Omega}|\nabla^{s}f(x)|},\frac
{F}{\left(  \frac{1}{|\Omega|}\int_{\Omega}|\nabla^{s}f(x)|^{2}dx\right)
^{1/2}}\right] \\
U=\left(  \lim\sup_{T\rightarrow\infty}\frac{1}{T}\int_{0}^{T}\frac{1}%
{|\Omega|}\int_{\Omega}|v(x,t)|^{2}dxdt\right)  ^{1/2}.
\end{array}
\right\}  \label{eq:FLUdefinition}%
\end{equation}
$L$ has units of length and satisfies%
\begin{equation}
||\nabla^{s}f||_{\infty}\leq\frac{F}{L}\text{ and }\frac{1}{|\Omega|}%
||\nabla^{s}f||^{2}\leq\frac{F^{2}}{L^{2}}\text{ }. \label{eq:PropertiesL}%
\end{equation}
We assume that weak solutions of the system satisfy the following energy
inequality.
\begin{equation}
\frac{d}{dt}\left(  \frac{1}{2}||v||^{2}+\int_{\Omega}kdx\right)
+2\nu||\nabla^{s}v||^{2}+\frac{\sqrt{2}}{2\tau}\int_{\Omega}kdx\leq(f,v).
\label{eq:EnergyINeq}%
\end{equation}
This is unproven for the new model but consistent with what is known for the
standard model, e.g., \cite{CL}. We assume the following energy equality for
the separate $k-$equation.
\begin{equation}
\frac{d}{dt}\int_{\Omega}kdx+\frac{\sqrt{2}}{2\tau}\int_{\Omega}%
kdx=\int_{\Omega}\sqrt{2}\mu k\tau|\nabla^{s}v|^{2}dx. \label{Eq:kEqnEnergyEQ}%
\end{equation}
This follows from the definition of a distributional solution by taking the
test function to be $\phi(x)\equiv1$.

\section{Proof that Condition 4 holds}

This section presents a proof that Condition 4 holds for the model
(\ref{eq:New1EqnModel}). The first steps of the proof parallel the\ estimates
in the NSE case in, e.g., \cite{DC92}, \cite{DF02}. With the above compressed
notation, the assumed model energy inequality, motivated by
(\ref{eq:EnergyINeq}), can be written%
\[
\frac{d}{dt}\left(  \frac{1}{2|\Omega|}||v||^{2}+\frac{1}{|\Omega|}%
\int_{\Omega}kdx\right)  +\frac{1}{|\Omega|}\int_{\Omega}2\nu|\nabla^{s}%
v|^{2}+\frac{\sqrt{2}}{2\tau}kdx\leq\frac{1}{|\Omega|}(f,v(t)).
\]
In the introduction the following uniform in $T$ bounds were proven%
\begin{equation}
\left.
\begin{array}
[c]{c}%
\frac{1}{2}||v(T)||^{2}+\int_{\Omega}k(T)dx\leq C<\infty\text{ ,}\\
\frac{1}{T}\int_{0}^{T}\int_{\Omega}\left(  2\nu|\nabla^{s}v|^{2}+\frac
{\sqrt{2}}{2\tau}k\right)  dxdt\leq C<\infty.
\end{array}
\right\}  \label{eq:aPrioriBounds}%
\end{equation}
Time averaging over $0<t<T$ gives%
\begin{gather*}
\frac{1}{T}\left(  \frac{1}{2}||v(T)||^{2}+\int_{\Omega}k(x,T)dx-\frac{1}%
{2}||v(0)||^{2}-\int_{\Omega}k(x,0)dx\right)  +\\
+\frac{1}{T}\int_{0}^{T}\int_{\Omega}\left(  2\nu|\nabla^{s}v|^{2}+\frac
{\sqrt{2}}{2\tau}k\right)  dxdt=\frac{1}{T}\int_{0}^{T}(f,v(t))dt.
\end{gather*}
In view of the \'{a} priori bounds (\ref{eq:aPrioriBounds}) and the
Cauchy-Schwarz inequality, this implies%
\begin{equation}
\mathcal{O}\left(  \frac{1}{T}\right)  +\frac{1}{T}\int_{0}^{T}\varepsilon
_{\text{model}}(t)dt\leq F\left(  \frac{1}{T}\int_{0}^{T}\frac{1}{|\Omega
|}||v||^{2}dt\right)  ^{\frac{1}{2}}. \label{eq:FirstStep}%
\end{equation}

To bound $F$ in terms of flow quantities, take the $L^{2}(\Omega)$ inner
product of (\ref{eq:New1EqnModel}) with $f(x)$, integrate by parts (i.e.,
select the test function to be $f(x)$ in the variational formulation) and
average over $[0,T]$. This gives%
\begin{gather}
F^{2}=\frac{1}{T}\frac{1}{|\Omega|}(v(T)-v_{0},f)-\frac{1}{T}\int_{0}^{T}%
\frac{1}{|\Omega|}(vv,\nabla^{s}f)dt+\\
+\frac{1}{T}\int_{0}^{T}\frac{1}{|\Omega|}\int_{\Omega}2\nu\nabla^{s}%
v:\nabla^{s}f+\sqrt{2}\mu k\tau\nabla^{s}v:\nabla^{s}fdxdt.\nonumber
\end{gather}
The \textbf{first term} on the RHS is\ $\mathcal{O}(1/T)$ as above. The
\textbf{second term} is bounded by the Cauchy-Schwarz inequality and
(\ref{eq:PropertiesL}). For any $0<\beta<1$
\begin{gather*}
\text{\textbf{Second:} }\left\vert \frac{1}{T}\int_{0}^{T}\frac{1}{|\Omega
|}(vv,\nabla^{s}f)dt\right\vert \leq\frac{1}{T}\int_{0}^{T}||\nabla^{s}%
f(\cdot)||_{\infty}\frac{1}{|\Omega|}||vv||^{2}dt\\
\leq||\nabla^{s}f(\cdot)||_{\infty}\frac{1}{T}\int_{0}^{T}\frac{1}{|\Omega
|}||v(\cdot,t)||^{2}dt\leq\frac{F}{L}\frac{1}{T}\int_{0}^{T}\frac{1}{|\Omega
|}||v(\cdot,t)||^{2}dt.
\end{gather*}
The \textbf{third term} is bounded by analogous steps to the second term. For
any $0<\beta<1$
\begin{gather*}
\text{\textbf{Third:} }\frac{1}{T}\int_{0}^{T}\frac{1}{|\Omega|}\int_{\Omega
}2\nu\nabla^{s}v(x,t):\nabla^{s}f(x)dxdt\leq\\
\leq\left(  \frac{1}{T}\int_{0}^{T}\frac{4\nu^{2}}{|\Omega|}||\nabla
^{s}v||^{2}dt\right)  ^{\frac{1}{2}}\left(  \frac{1}{T}\int_{0}^{T}\frac
{1}{|\Omega|}||\nabla^{s}f||^{2}dt\right)  ^{\frac{1}{2}}\\
\leq\left(  \frac{1}{T}\int_{0}^{T}\frac{2\nu}{|\Omega|}||\nabla^{s}%
v||^{2}dt\right)  ^{\frac{1}{2}}\frac{\sqrt{2\nu}F}{L}\leq\frac{\beta F}%
{2U}\frac{1}{T}\int_{0}^{T}\frac{2\nu}{|\Omega|}||\nabla^{s}v||^{2}dt+\frac
{1}{\beta}\frac{\nu UF}{L^{2}}.
\end{gather*}
The \textbf{fourth term} is model specific. Its estimation begins by
successive applications of the space then time Cauchy-Schwarz inequality as
follows%
\begin{gather*}
\text{\textbf{Fourth}: }\left\vert \frac{1}{T}\int_{0}^{T}\frac{1}{|\Omega
|}\int_{\Omega}\sqrt{2}\mu k\tau\nabla^{s}v(x,t):\nabla^{s}f(x)dxdt\right\vert
\leq\\
\leq\frac{1}{T}\int_{0}^{T}\frac{1}{|\Omega|}\int_{\Omega}\left(  \sqrt
{\sqrt{2}\mu k\tau}\right)  \left(  \sqrt{\sqrt{2}\mu k\tau}|\nabla
^{s}v|\right)  |\nabla^{s}f|dxdt\\
\leq||\nabla^{s}f||_{\infty}\frac{1}{T}\int_{0}^{T}\left(  \frac{1}{|\Omega
|}\int_{\Omega}\sqrt{2}\mu k\tau dx\right)  ^{\frac{1}{2}}\left(  \frac
{1}{|\Omega|}\int_{\Omega}\sqrt{2}\mu k\tau|\nabla^{s}v|^{2}dx\right)
^{\frac{1}{2}}dxdt\\
\leq\frac{F}{L}\left(  \frac{U}{FT}\int_{0}^{T}\frac{1}{|\Omega|}\int_{\Omega
}\sqrt{2}\mu k\tau dxdt\right)  ^{\frac{1}{2}}\left(  \frac{F}{UT}\int_{0}%
^{T}\frac{1}{|\Omega|}\int_{\Omega}\sqrt{2}\mu k\tau|\nabla^{s}v|^{2}%
dxdt\right)  ^{\frac{1}{2}}.
\end{gather*}
The arithmetic-geometric mean inequality then implies%
\begin{gather*}
\text{\textbf{Fourth}: }\left\vert \frac{1}{T}\int_{0}^{T}\frac{1}{|\Omega
|}\int_{\Omega}\sqrt{2}\mu k\tau\nabla^{s}v(x,t):\nabla^{s}f(x)dxdt\right\vert
\leq\\
\leq\frac{\beta}{2}\frac{F}{UT}\int_{0}^{T}\frac{1}{|\Omega|}\int_{\Omega
}\sqrt{2}\mu k\tau|\nabla^{s}v|^{2}dxdt+\frac{U}{2\beta F}\frac{F^{2}}{L^{2}%
}\frac{1}{T}\int_{0}^{T}\frac{1}{|\Omega|}\int_{\Omega}\sqrt{2}\mu k\tau
dxdt\\
\leq\frac{\beta}{2}\frac{F}{UT}\int_{0}^{T}\frac{1}{|\Omega|}\int_{\Omega
}\sqrt{2}\mu k\tau|\nabla^{s}v|^{2}dxdt+\frac{1}{2\beta}\frac{UF}{L^{2}T}%
\int_{0}^{T}\frac{1}{|\Omega|}\int_{\Omega}\sqrt{2}\mu k\tau dxdt.
\end{gather*}
Using these four estimates in the bound for $F^{2}$ yields%
\begin{align*}
F^{2} &  \leq\mathcal{O}\left(  \frac{1}{T}\right)  +\frac{F}{L}\frac{1}%
{T}\int_{0}^{T}\frac{1}{|\Omega|}||v||^{2}dt+\frac{1}{2\beta}\frac{UF}{L^{2}%
}\frac{1}{T}\int_{0}^{T}\frac{1}{|\Omega|}\int_{\Omega}\sqrt{2}\mu k\tau
dxdt\\
&  +\frac{1}{\beta}\frac{\nu UF}{L^{2}}+\frac{\beta F}{2U}\frac{1}{T}\int
_{0}^{T}\frac{1}{|\Omega|}\int_{\Omega}\left[  2\nu+\sqrt{2}\mu k\tau\right]
|\nabla^{s}v|^{2}dxdt.
\end{align*}
Thus, we have an estimate for $F\left(  \frac{1}{T}\int_{0}^{T}\frac
{1}{|\Omega|}||v||^{2}dt\right)  ^{\frac{1}{2}}:$%
\begin{gather*}
F\left(  \frac{1}{T}\int_{0}^{T}\frac{1}{|\Omega|}||v||^{2}dt\right)
^{\frac{1}{2}}\leq\mathcal{O}\left(  \frac{1}{T}\right)  +\frac{1}{L}\left(
\frac{1}{T}\int_{0}^{T}\frac{1}{|\Omega|}||v||^{2}dt\right)  ^{\frac{3}{2}}+\\
+\frac{\beta}{2}\frac{\left(  \frac{1}{T}\int_{0}^{T}\frac{1}{|\Omega
|}||v||^{2}dt\right)  ^{\frac{1}{2}}}{U}\frac{1}{T}\int_{0}^{T}\frac
{1}{|\Omega|}\int_{\Omega}\left[  2\nu+\sqrt{2}\mu k\tau\right]  |\nabla
^{s}v|^{2}dxdt+\\
+\frac{1}{2\beta}\left(  \frac{1}{T}\int_{0}^{T}\frac{1}{|\Omega|}%
||v||^{2}dt\right)  ^{\frac{1}{2}}\frac{2\nu U}{L^{2}}+\\
+\frac{1}{2\beta}\left(  \frac{1}{T}\int_{0}^{T}\frac{1}{|\Omega|}%
||v||^{2}dt\right)  ^{\frac{1}{2}}\frac{U}{L^{2}}\frac{1}{T}\int_{0}^{T}%
\frac{1}{|\Omega|}\int_{\Omega}\sqrt{2}\mu k\tau dxdt.
\end{gather*}
Inserting this on the RHS of (\ref{eq:FirstStep}) yields%
\begin{gather}
\frac{1}{T}\int_{0}^{T}\varepsilon_{\text{model}}dt\leq\mathcal{O}\left(
\frac{1}{T}\right)  +\frac{1}{L}\left(  \frac{1}{T}\int_{0}^{T}\frac
{1}{|\Omega|}||v||^{2}dt\right)  ^{\frac{3}{2}}+\label{eq:SecondStep}\\
+\frac{\beta}{2}\frac{\left(  \frac{1}{T}\int_{0}^{T}\frac{1}{|\Omega
|}||v||^{2}dt\right)  ^{\frac{1}{2}}}{U}\frac{1}{T}\int_{0}^{T}\frac
{1}{|\Omega|}\int_{\Omega}\left[  2\nu+\sqrt{2}\mu k\tau\right]  |\nabla
^{s}v|^{2}dxdt+\nonumber\\
+\frac{1}{2\beta}\left(  \frac{1}{T}\int_{0}^{T}\frac{1}{|\Omega|}%
||v||^{2}dt\right)  ^{\frac{1}{2}}U\frac{2\nu}{L^{2}}+\nonumber\\
+\frac{1}{2\beta}\left(  \frac{1}{T}\int_{0}^{T}\frac{1}{|\Omega|}%
||v||^{2}dt\right)  ^{\frac{1}{2}}\frac{U}{L^{2}}\left(  \frac{1}{T}\int
_{0}^{T}\frac{1}{|\Omega|}\int_{\Omega}\sqrt{2}\mu k\tau dxdt\right)
.\nonumber
\end{gather}
We prove in the next lemma an estimate for the last, model specific, term
$\int\sqrt{2}\mu k\tau dx$ on the RHS. This estimate has the interpretation
that, on time average, the decay (relaxation) rate of $k(x,t)$ balances the
transfer rate of kinetic energy from means to fluctuations.

\begin{lemma}
For weak solutions of the $k-$equation we have%
\[
\left\langle \frac{1}{|\Omega|}\int_{\Omega}\sqrt{2}\mu k(x,t)\tau
dx\right\rangle =2\mu\tau^{2}\left\langle \frac{1}{|\Omega|}\int_{\Omega}%
\sqrt{2}\mu k\tau|\nabla^{s}v|^{2}dx\right\rangle .
\]

\end{lemma}

\begin{proof}
[of Lemma 1]Integrating the $k-$equation (i.e., choosing $\phi(x)\equiv1$ in
the equation's distributional formulation) yields%
\[
\frac{d}{dt}\frac{1}{|\Omega|}\int_{\Omega}kdx+\frac{\sqrt{2}}{2\tau}\frac
{1}{|\Omega|}\int_{\Omega}kdx=\frac{1}{|\Omega|}\int_{\Omega}\sqrt{2}\mu
k\tau|\nabla^{s}v|^{2}dx.
\]
From Theorem 1, $\int kdx$ (and thus its time averages) is uniformly bounded
in time. Thus, we can time average the above. This gives%
\begin{gather*}
\mathcal{O}\left(  \frac{1}{T}\right)  +\frac{\sqrt{2}}{2\tau}\frac{1}{T}%
\int_{0}^{T}\frac{1}{|\Omega|}\int_{\Omega}kdxdt=\frac{1}{T}\int_{0}^{T}%
\frac{1}{|\Omega|}\int_{\Omega}\sqrt{2}\mu k\tau|\nabla^{s}v|^{2}dxdt,\text{
}\\
\text{and thus}\\
\left\langle \frac{1}{|\Omega|}\int_{\Omega}\sqrt{2}\mu k(x,t)\tau
dx\right\rangle =2\mu\tau^{2}\left\langle \frac{1}{|\Omega|}\int_{\Omega}%
\sqrt{2}\mu k\tau|\nabla^{s}v|^{2}dx\right\rangle ,
\end{gather*}
proving the lemma.
\end{proof}

To continue the proof of Theorem 1, this lemma is now used to replace terms on
the RHS of (\ref{eq:SecondStep}) involving $\sqrt{2}\mu k\tau|\nabla^{s}%
v|^{2}$\ by terms with $\sqrt{2}\mu k(x,t)\tau$. Let $T_{j}\rightarrow\infty$
in (\ref{eq:SecondStep}), recalling the definition of\ $\varepsilon
_{\text{model}}$ and inserting the above relation for the last term yields%
\begin{gather}
\left\langle \frac{1}{|\Omega|}\int_{\Omega}\left[  2\nu|\nabla^{s}%
v(x,t)|^{2}+\frac{\sqrt{2}}{2}\tau^{-1}k(x,t)\right]  dx\right\rangle
\leq\frac{U^{3}}{L}+\\
+\frac{\beta}{2}\left\langle \frac{1}{|\Omega|}\int_{\Omega}2\nu|\nabla
^{s}v|^{2}+\frac{1}{2\mu\tau^{2}}\sqrt{2}\mu k(x,t)\tau dx\right\rangle
+\nonumber\\
+\frac{1}{\beta}U^{2}\frac{\nu}{L^{2}}+\frac{1}{2\beta}\frac{U^{2}}{L^{2}%
}\left\langle \frac{1}{|\Omega|}\int_{\Omega}\sqrt{2}\mu k(x,t)\tau
dx\right\rangle .\nonumber
\end{gather}
Collecting terms gives%
\begin{gather}
\left\langle \frac{1}{|\Omega|}\int_{\Omega}\left[  2\nu|\nabla^{s}%
v(x,t)|^{2}+\frac{\sqrt{2}}{2}\tau^{-1}k(x,t)\right]  dx\right\rangle
\leq\frac{1}{L}U^{3}+\frac{1}{\beta}U^{2}\frac{\nu}{L^{2}}\\
+\frac{\beta}{2}\left\langle \frac{1}{|\Omega|}\int_{\Omega}2\nu|\nabla
^{s}v|^{2}+\left(  \frac{1}{2\mu\tau^{2}}+\frac{1}{2\beta}\frac{U^{2}}{L^{2}%
}\right)  \sqrt{2}\mu k(x,t)\tau dx\right\rangle .\nonumber
\end{gather}
The multiplier of $\sqrt{2}\mu k(x,t)\tau$ simplifies to
\[
\frac{\beta}{2}\left(  \frac{1}{2\mu\tau^{2}}+\frac{1}{2\beta}\frac{U^{2}%
}{L^{2}}\right)  \sqrt{2}\mu\tau=\frac{\sqrt{2}}{2}\tau^{-1}\left[
\frac{\beta}{2}+\frac{1}{2}\mu\frac{U^{2}}{L^{2}}\tau^{2}\right]  .
\]
Thus, rearrange the above inequality to read%
\begin{gather*}
\left\langle \frac{1}{|\Omega|}\int_{\Omega}\left[  \left(  1-\frac{\beta}%
{2}\right)  \nu|\nabla^{s}v|^{2}+\left(  1-\left\{  \frac{\beta}{2}+\frac{\mu
}{2}\frac{U^{2}}{L^{2}}\tau^{2}\right\}  \right)  \frac{\sqrt{2}}{2}\tau
^{-1}k\right]  dx\right\rangle \\
\leq\frac{U^{3}}{L}+\frac{1}{\beta}U^{2}\frac{\nu}{L^{2}}=\left(  1+\frac
{1}{\beta}\mathcal{R}e^{-1}\right)  \frac{U^{3}}{L}.
\end{gather*}
Pick (without optimizing) $\beta=1$. This yields
\begin{align*}
&  \left\langle \frac{1}{|\Omega|}\int_{\Omega}\left[  \nu|\nabla
^{s}v(x,t)|^{2}+\frac{\sqrt{2}}{2}\tau^{-1}k(x,t)\right]  dx\right\rangle \\
&  \leq\frac{2}{\min\{1,1-\mu\frac{U^{2}}{L^{2}}\tau^{2}\}}\left\{
\frac{U^{3}}{L}+\mathcal{R}e^{-1}\frac{U^{3}}{L}\right\}  .
\end{align*}
We clearly desire
\[
1-\mu\frac{U^{2}}{L^{2}}\tau^{2}=1-\mu\left(  \frac{\tau}{T^{\ast}}\right)
^{2}\geq\frac{1}{2}.
\]
This holds if the time cutoff $\tau$ is chosen with respect to the global
turnover time $T^{\ast}=L/U$\ so that%
\[
\frac{\tau}{T^{\ast}}\leq\sqrt{\frac{1}{\mu}}\simeq1.35,\text{ for }\mu=0.55.
\]
Then we have, as claimed,%
\[
\left\langle \frac{1}{|\Omega|}\int_{\Omega}\left[  \nu|\nabla^{s}v|^{2}%
+\frac{\sqrt{2}}{2}\tau^{-1}k\right]  dx\right\rangle \leq4\left(
1+\mathcal{R}e^{-1}\right)  \frac{U^{3}}{L}.
\]

\section{Numerical illustrations in 2d and 3d}

This section shows that the static and kinematic turbulence length scales
produces flows with different statistics. We use the simplest reasonable
choices
\[
l_{0}(x)=\min\{0.41y,0.41\cdot0.2\mathcal{R}e^{-1/2}\}\text{ \ and \ }%
l_{K}(x,t)=\sqrt{2}k(x,t)^{1/2}\tau.
\]
All numerical experiments were performed using the package FEniCS. We consider
several normalized, space-averaged statistics. Recall that the
\textit{turbulence intensity} is $I=\left\langle ||u^{\prime}||^{2}%
\right\rangle /\left\langle ||\overline{u}||^{2}\right\rangle $. An
approximation to the (time) evolution of this is calculable from the model%
\[
I_{\text{model}}(t):=\frac{\frac{2}{|\Omega|}\int_{\Omega}k(x,t)dx}{\frac
{1}{|\Omega|}\int_{\Omega}|v(x,t)|^{2}dx}.
\]
%

Next we consider the effective viscosity coefficient for the two methods. The
\textit{effective viscosity} is a useful statistic to quantify the aggregate,
space averaged effect of fluctuating eddy viscosity terms. It is%
\[
\nu_{\text{effective}}(t):=\frac{\frac{1}{|\Omega|}\int_{\Omega}\left[
\nu+\mu l\sqrt{k}\right]  |\nabla^{s}v|^{2}dx}{\frac{1}{|\Omega|}\int_{\Omega
}|\nabla^{s}v|^{2}dx}.
\]
We also consider the related statistic of the \textit{viscosity ratio of
turbulent viscosity to molecular viscosity}
\[
VR(t):=\frac{\frac{1}{|\Omega|}\int_{\Omega}\mu l\sqrt{k}|\nabla^{s}v|^{2}%
dx}{\frac{1}{|\Omega|}\int_{\Omega}2\nu|\nabla^{s}v|^{2}dx}.
\]
We also calculate the evolution of the \textit{Taylor microscale} of each
model's solution:%
\[
\lambda_{\text{Taylor}}(t):=\left(  \frac{\int_{\Omega}|\nabla^{s}v|^{2}%
dt}{\int_{\Omega}|v|^{2}dt}\right)  ^{-1/2}.
\]
The time evolution of the \textit{scaled averaged turbulence length scale} and
turbulent viscosity are also of interest:%
\begin{align*}
\frac{avg(l)}{L} &  :=\frac{1}{L}\left(  \frac{1}{|\Omega|}\int_{\Omega
}l(x,t)^{2}dx\right)  ^{1/2}\\
\frac{avg(\nu_{T})}{LU} &  :=\frac{1}{LU}\frac{1}{|\Omega|}\int_{\Omega}\mu
l(x,t)\sqrt{k(x,t)}dx.
\end{align*}

\subsection{Test 1: Flow between 2d offset circles}

For the first test, we consider a two-dimensional rotational flow obstructed
by a circular obstacle with no-slip boundary conditions. Let $\Omega_{1}
\subset\mathbb{R}^{2}$, where
\[
\Omega_{1}=\{(x,y)\in\mathbb{R}^{2}:x^{2}+y^{2}<1\}\setminus\{(x,y)\in
\mathbb{R}^{2}:(x-.5)^{2}+y^{2}\leq.01\}.
\]
The domain $\Omega_{1}$ is discretized via a Delaunay triangulation with a
maximal mesh width of $.01$; a plot is given below. From the plot in Figure 1
of the model's Taylor microscale this mesh fully resolves the model solution.
\begin{figure}[t]
\subfloat[$\Omega$]{\includegraphics[width=\textwidth]{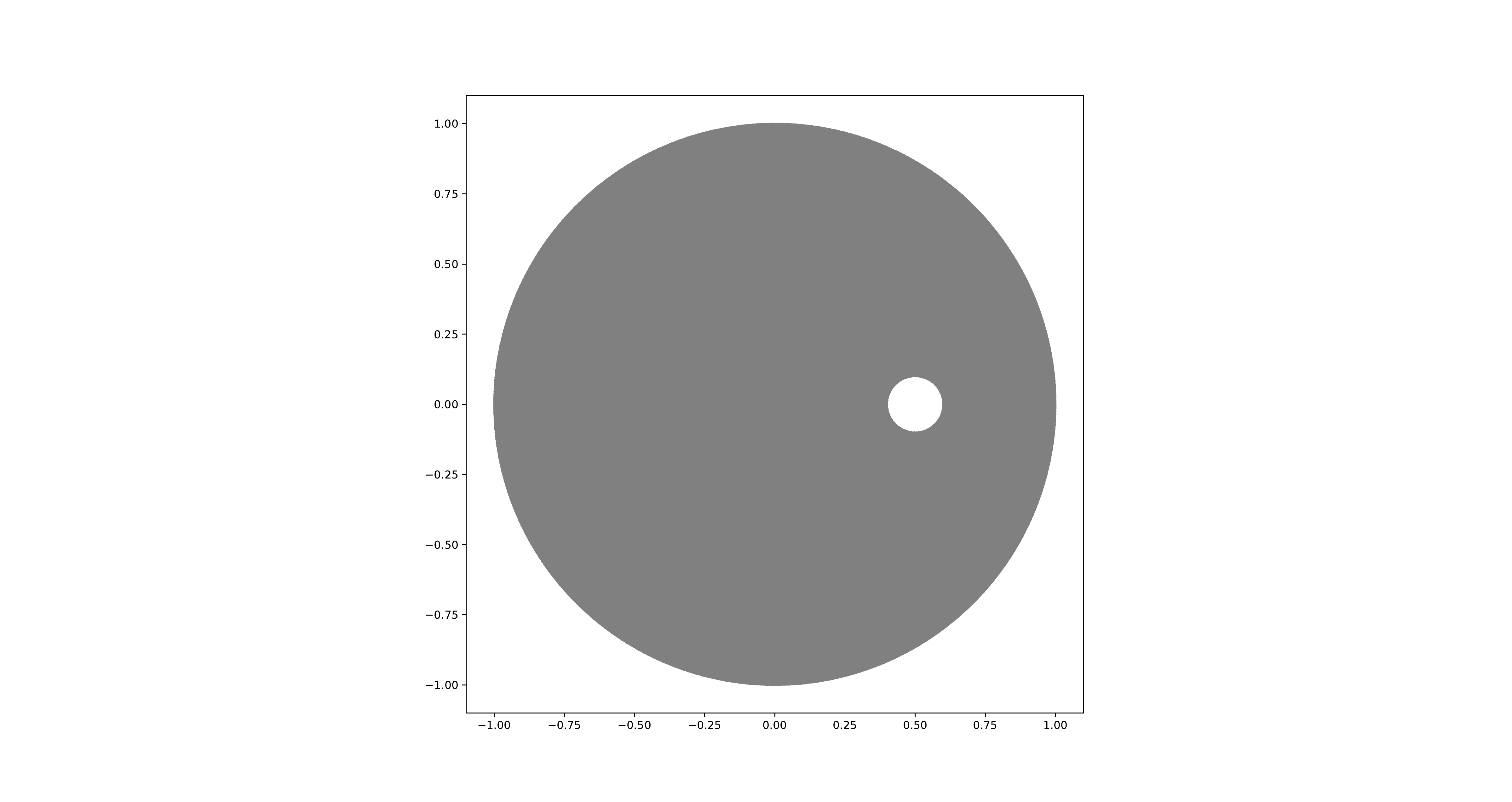}\label{figmesh:a}}\newline
\subfloat[$\Omega$ near the obstacle]{\includegraphics[width=\textwidth]{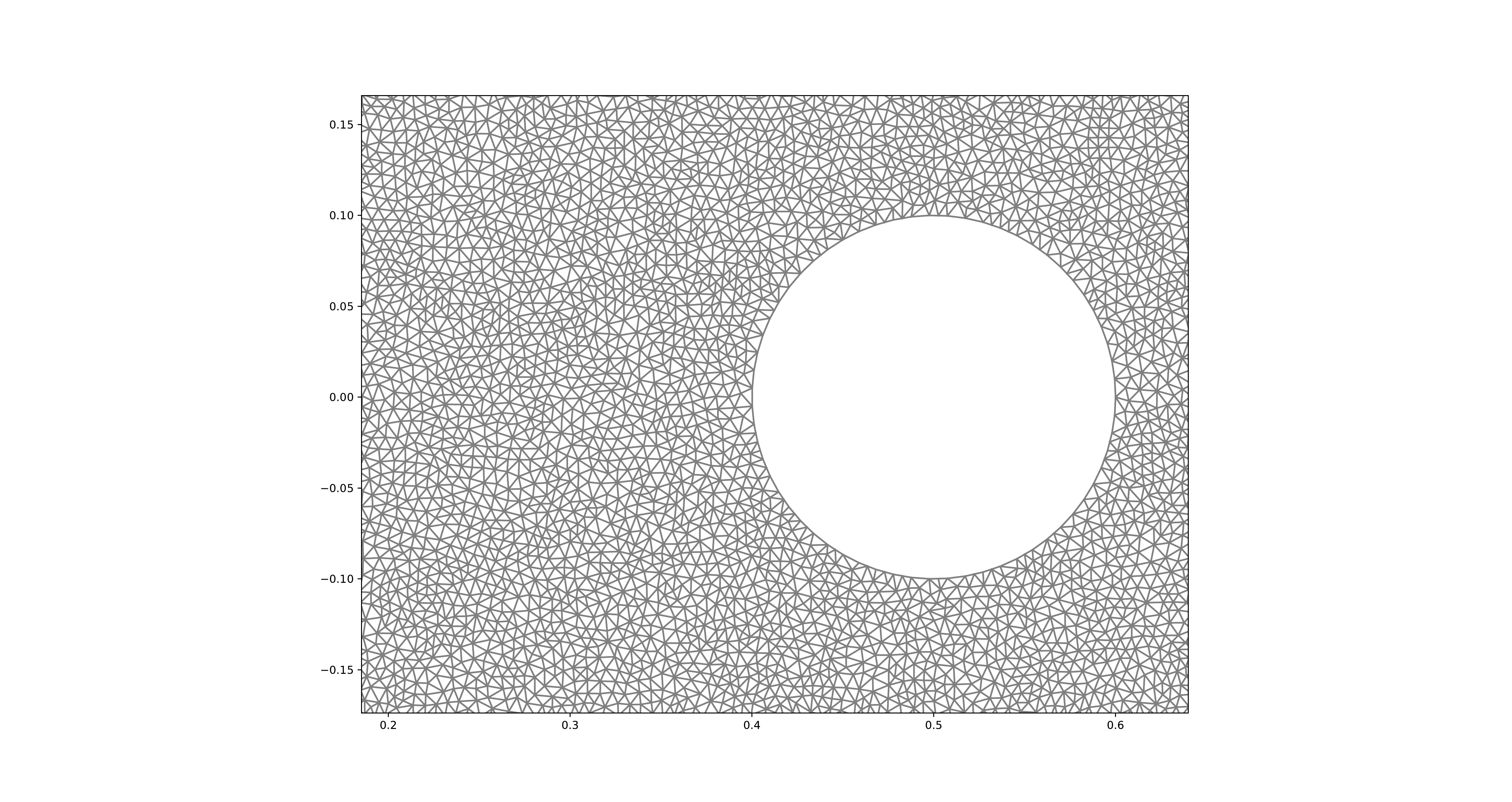}\label{figmesh:b}}
\caption{Discretization of $\Omega$}
\label{fig=mesh}
\end{figure}

We start the test at rest, i.e., $v_{0}=(0,0)^{T}$, and let the fluid have
kinematic viscosity $\nu=0.0001$. We take the final time $T=10$ and averaging
window $\tau=1$. Rather than give an interpretation of the time average for
$0\leq t<1$ we harvest flow statistics for $t\geq1$ after a cold start and
ramping up the body force with a multiplier $\min\{t,1\}$. To generate
counter-clockwise motion we impose the body force
\[
f(x,y;t)=\min\{t,1\}(-4y(1-x^{2}-y^{2}),4x(1-x^{2}-y^{2}))^{T}.
\]
\textbf{Initial Conditions.} An initial condition for the velocity, $v(x,0)$,
and for the TKE $k(x,0)$\ must be specified. For some flows standard choices
are known\footnote{For example, for turbulent flow in a square duct, a choice
is
\begin{gather*}
k(x,0)=1.5|u_{0}(x)|^{2}I^{2}\text{ where}\\
\text{ }I\text{ = turbulent intensity}\simeq0.16\mathcal{R}e^{-1/8}\text{ .}%
\end{gather*}
}. We use a different and systematic approach to the initial condition
$k(x,0)$\ as follows. From $l(x,t)=\sqrt{2}k^{1/2}\tau$ we set at $t=0,$
$l=l_{0}(x)$ and solve for $k(x,0)$. This yields the initial condition%
\[
k(x,0)=\frac{1}{2\tau^{2}}l_{0}^{2}(x)\text{ where }l_{0}(x)=\min
\{0.41y,0.082\mathcal{R}e^{-1/2}\}.
\]
This choice means that $l_{0}(x)=l_{K}(x,0)$.

To compare the models, we plot the temporal evolution of the above statistics.
For both models, we let $\mu=0.55$ and timestep $\Delta{t}=.01$. To let the
flow develop, we first activate both models when $t=1$. \begin{figure}[t]
\subfloat[Model intensity
$I_{\text{model}}$]{\includegraphics[width=.49\textwidth]{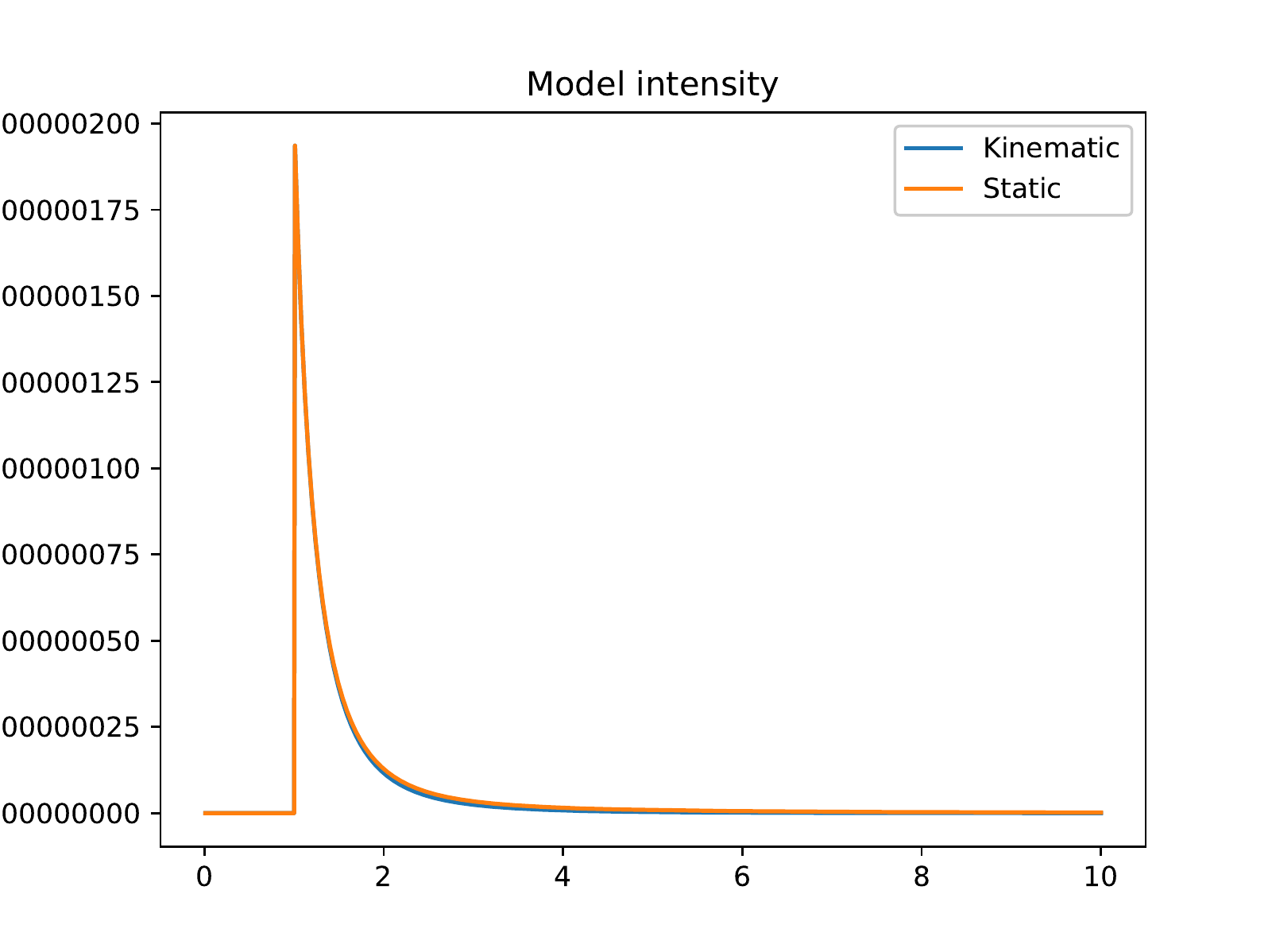}\label{fig2:a}}
\subfloat[Effective viscosity
$\nu_{\text{effective}}$]{\includegraphics[width=.49\textwidth]{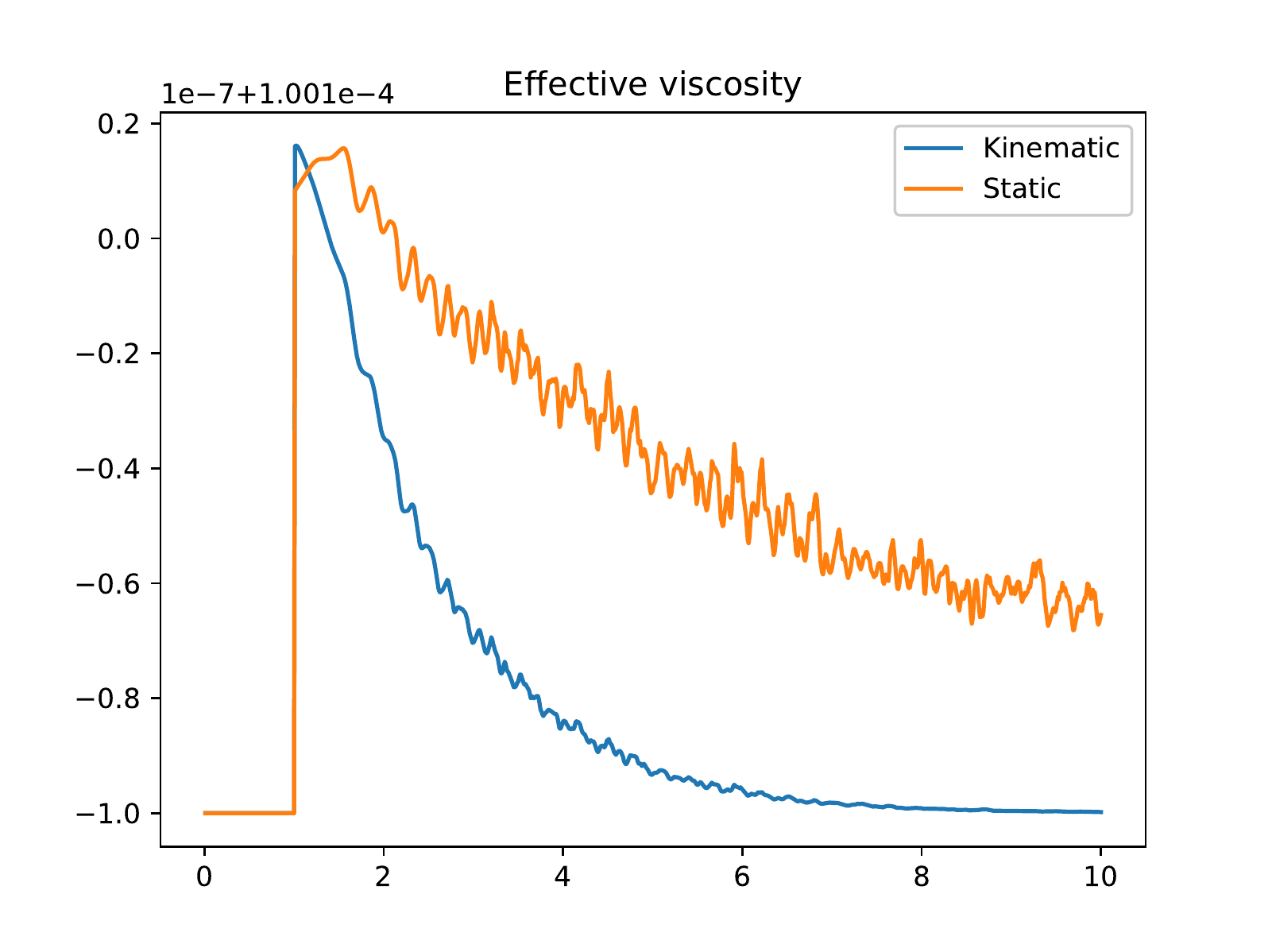}\label{fig2:b}}\newline
\subfloat[Viscosity ratio
$VR_1$]{\includegraphics[width=.49\textwidth]{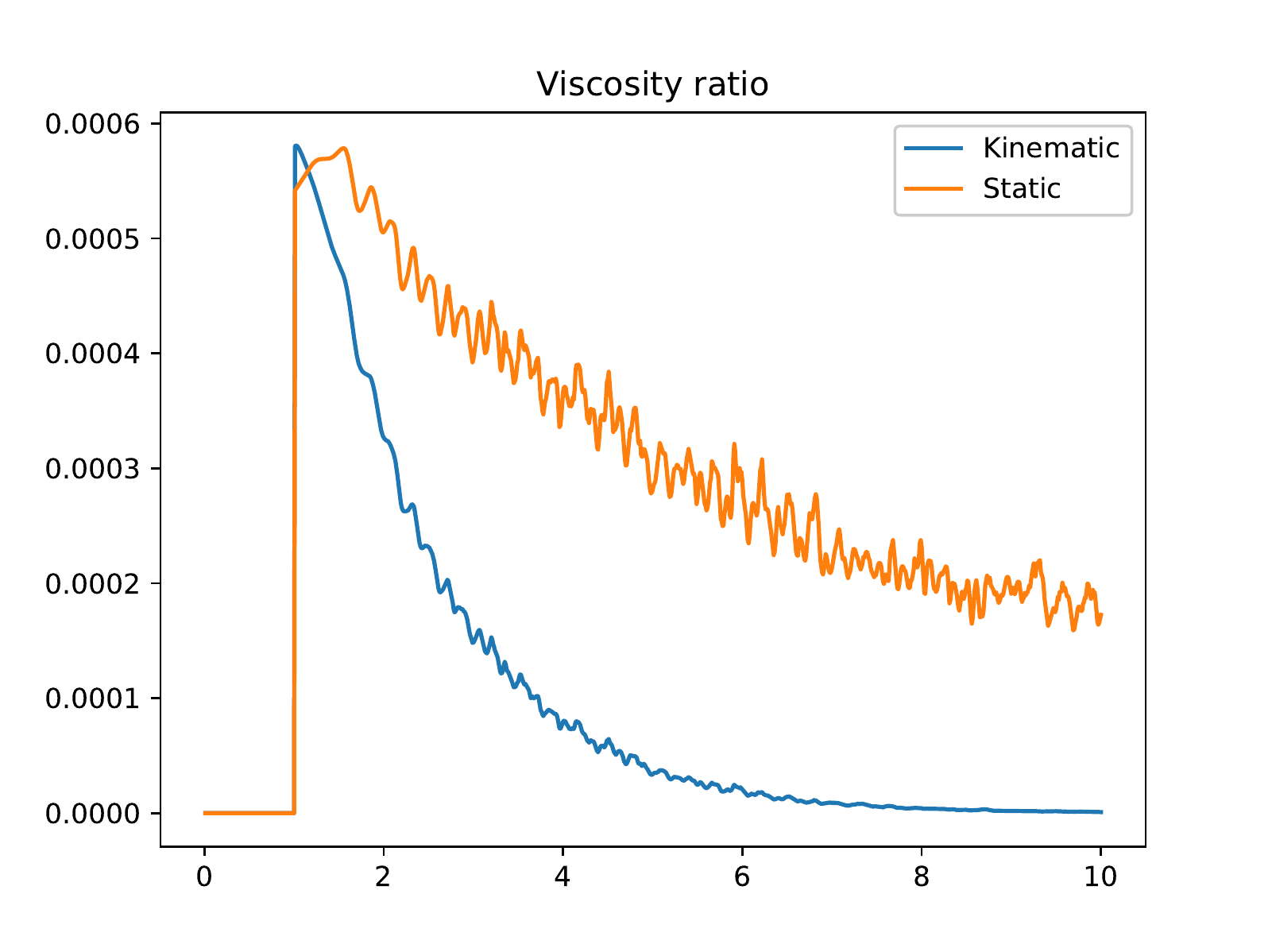}\label{fig2:c}}
\subfloat[Taylor microscale
$\lambda_{\text{Taylor}}$]{\includegraphics[width=.49\textwidth]{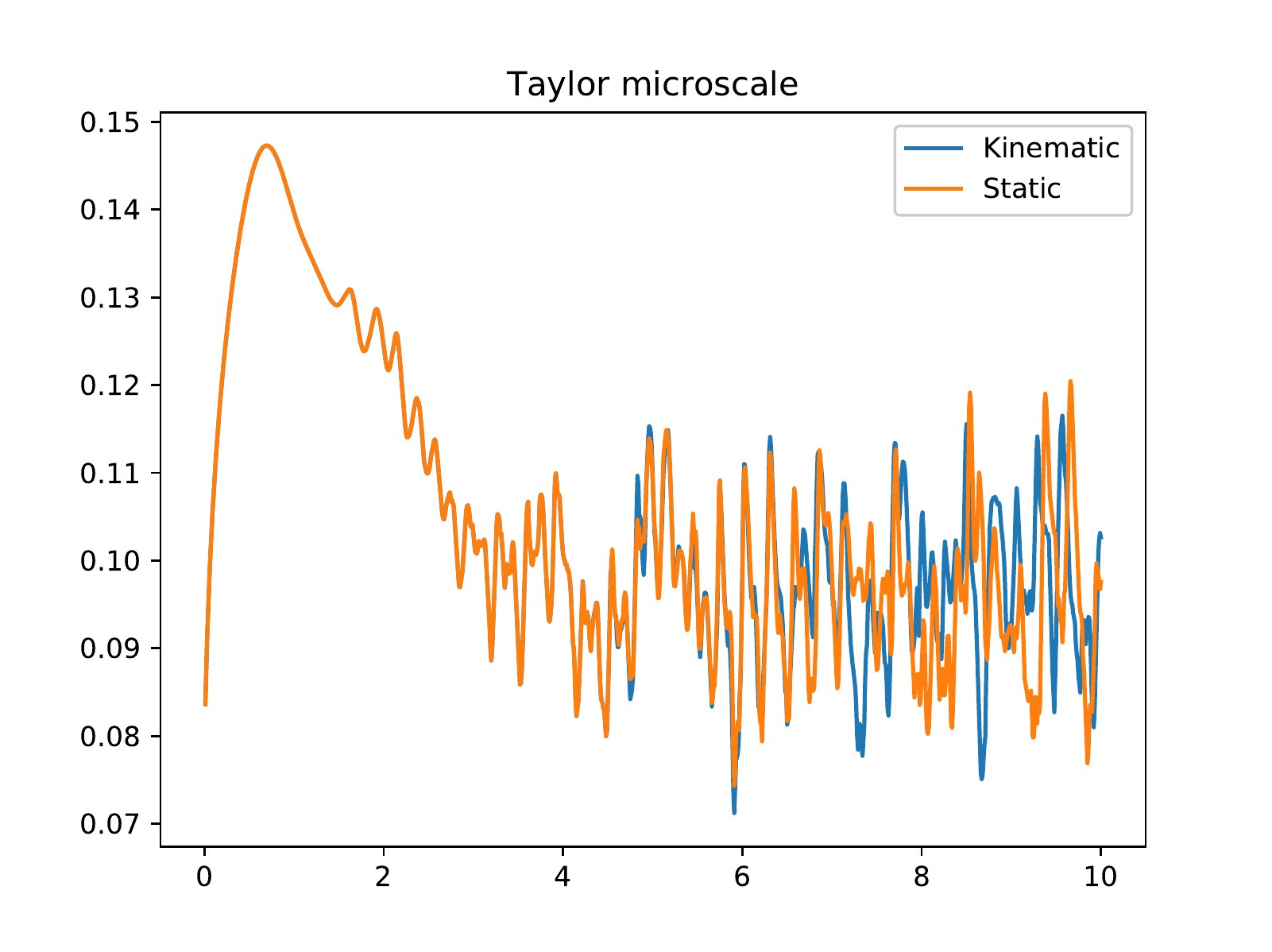}\label{fig2:d}}\newline
\subfloat[$avg{(l)}/L$]{\includegraphics[width=.49\textwidth]{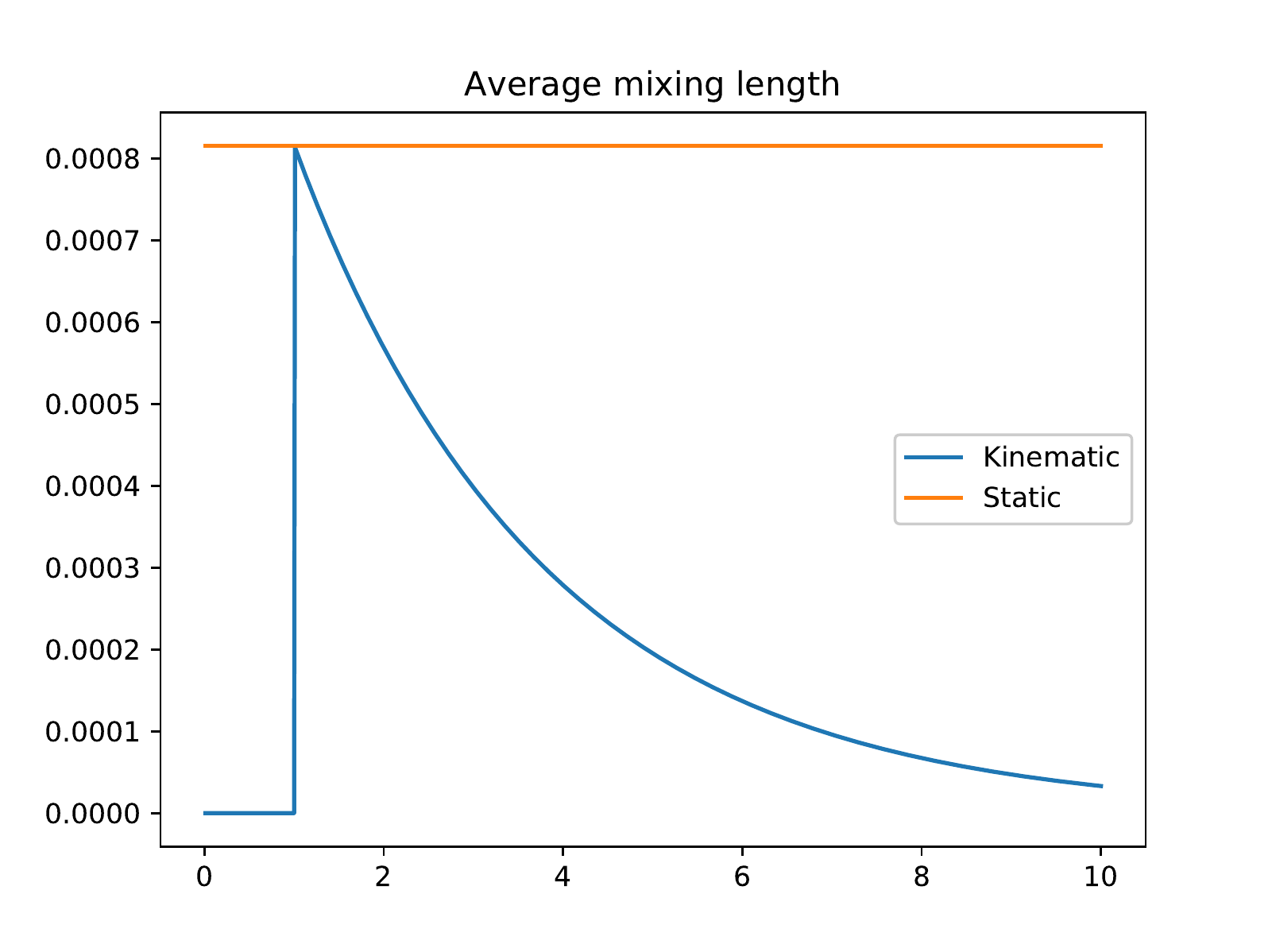}\label{fig2:e}}
\subfloat[$avg{(\nu_T)}/UL$]{\includegraphics[width=.49\textwidth]{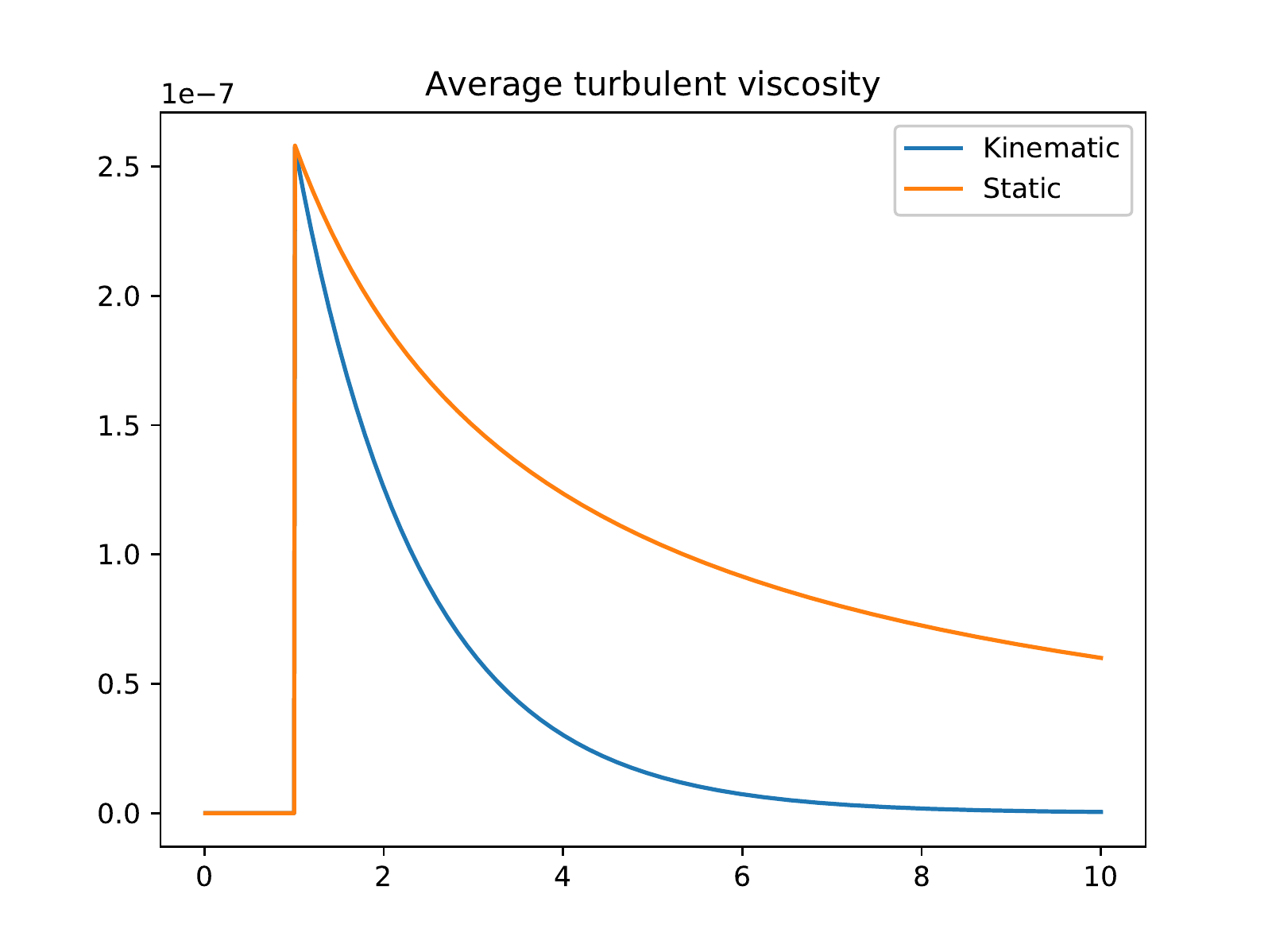}\label{fig2:f}}
\caption{2d Flow statistics for both models.}
\label{fig=stats_1}%
\end{figure}In the test, the model's estimate of the turbulent intensity for
both is similar, as shown in Figure \ref{fig2:a}. In \cite{JL14b} the
turbulent intensity was estimated by an ensemble simulation. For ensemble
averaging $I$ was significant larger than calculated here by time averaging
and with the $1-$equation model. Either intensities by time and ensemble
averaging do not coincide or $I_{model}$ is not an accurate turbulent
intensity. Figure \ref{fig2:b} shows that the effective viscosity for the
kinematic length scale is significantly smaller than for the standard model.
This is consistent with Figure \ref{fig2:c}, \ref{fig2:e} and \ref{fig2:f}. In
Figure \ref{fig2:d} the Taylor microscale is larger than expected, possibly
due to numerical dissipation in the fully implicit time discretization used.

The statistics considered reveal differences in the two models. Figure
\ref{fig2:b} shows that the kinematic model has an effective viscosity that
decays to $\nu_{\text{effective}}=0.0001$ more rapidly than does the static
model. More evidence of this fact is given in Figure \ref{fig2:c}, which shows
the turbulent-to-molecular viscosity ratio. The comparison of the evolution of
the Taylor microscale, given in Figure \ref{fig2:d}, shows similar profiles
until $t\approx5$. Figure \ref{fig2:e}, which compares the evolution of the
average mixing length, shows that the kinematic mixing length model decreases
the turbulence length scale over the course of the simulation. Finally, Figure
\ref{fig2:f} shows that the average turbulent viscosity for the kinematic
model is consistently smaller than that of the static model. Statistical
comparisons of both of these models with different parameters (in particular,
the turbulent time scale $\tau$) are also of interest. Below, we give semilog
(in the vertical axis) plots of the average mixing length with different
values of $\tau$. \begin{figure}[t]
\subfloat[$avg{(l)}/L$ for $\tau=.01,.1,1$.]{\includegraphics[width=.49\textwidth]{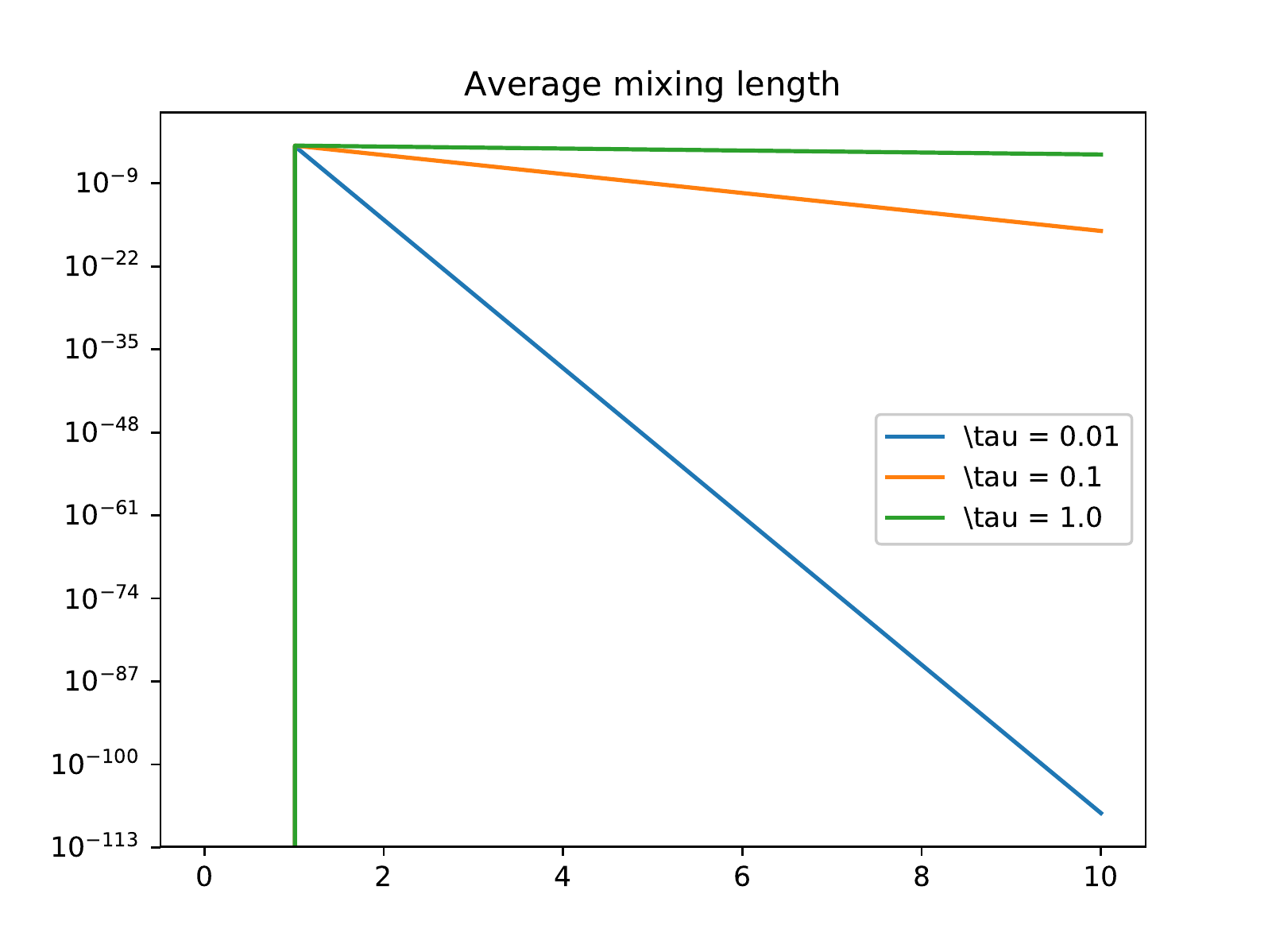}\label{figcomp:a}}
\subfloat[$avg{(l)}/L$ for $\tau=1,10,100$ and the static model.]{\includegraphics[width=.49\textwidth]{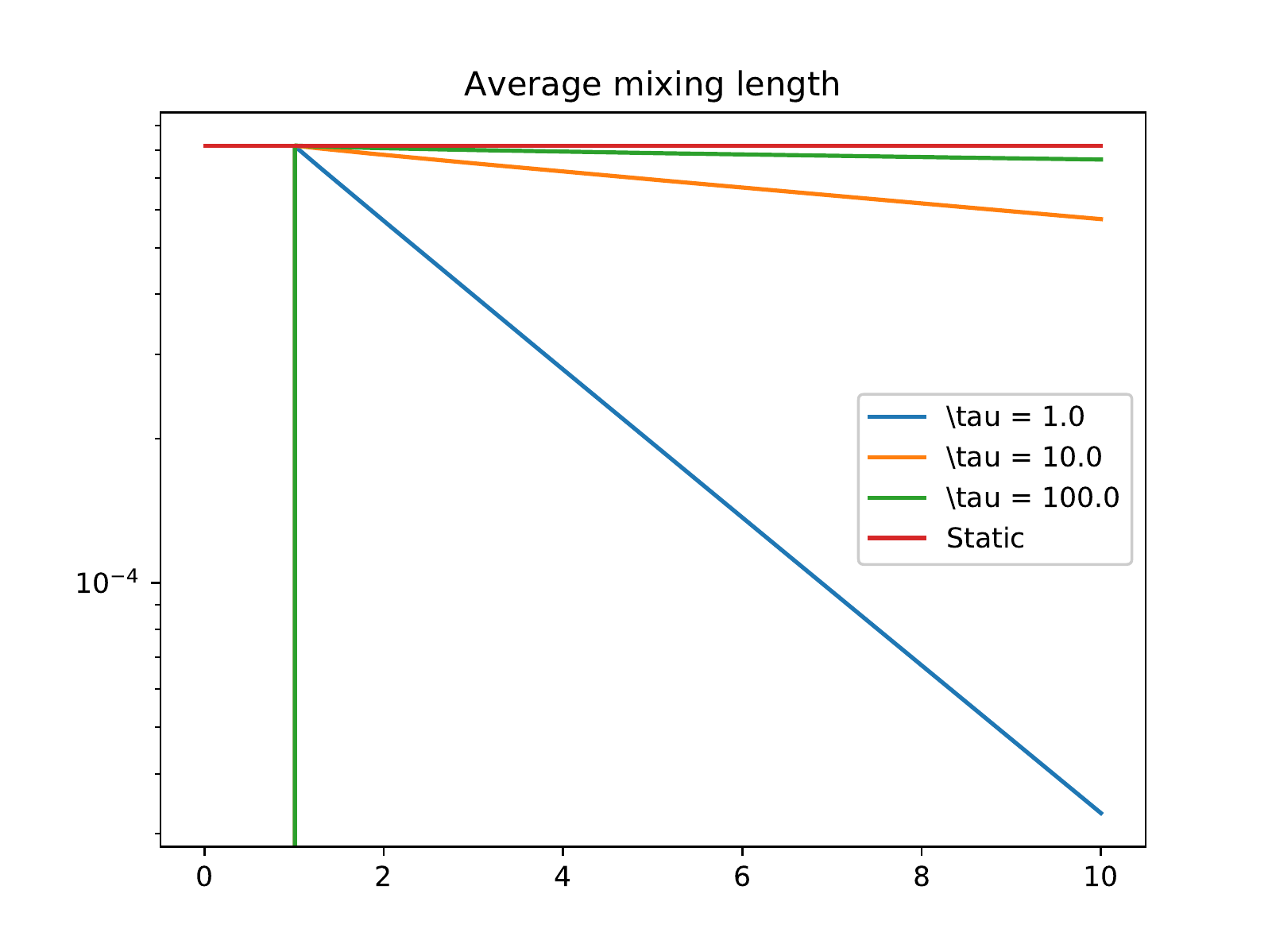}\label{figcomp:b}}
\caption{Average mixing length comparison}
\label{fig=comp}
\end{figure}Figure \ref{fig=comp} shows that decreasing values of $\tau$ lead
to a vanishing average mixing length, whereas increasing $\tau$ yields average
mixing lengths that appear to converge to the static mixing length.

Next, we give plots of the velocity magnitude and squared vorticity for the
kinematic model at $t=1,5,$ and $10$. \begin{figure}[t]
\subfloat[Velocity
$(t=1)$]{\includegraphics[width=.49\textwidth]{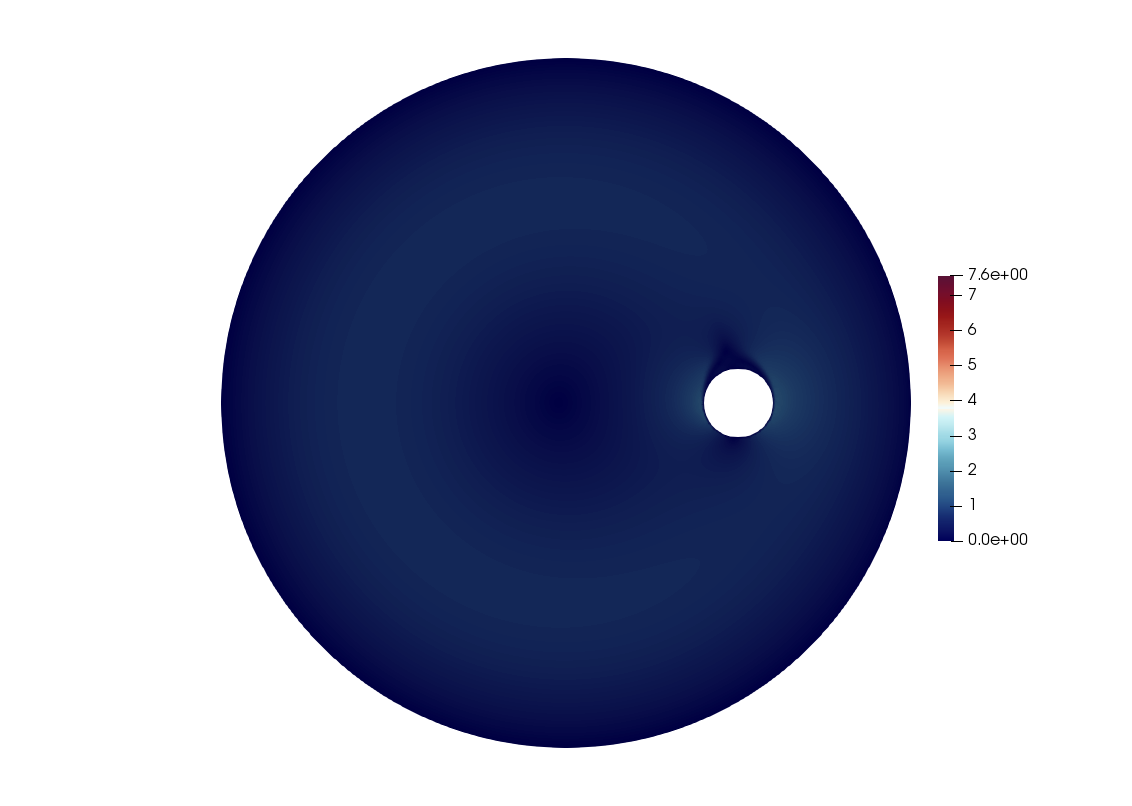}\label{fig31:a}}
\subfloat[Squared vorticity
$(t=1)$]{\includegraphics[width=.49\textwidth]{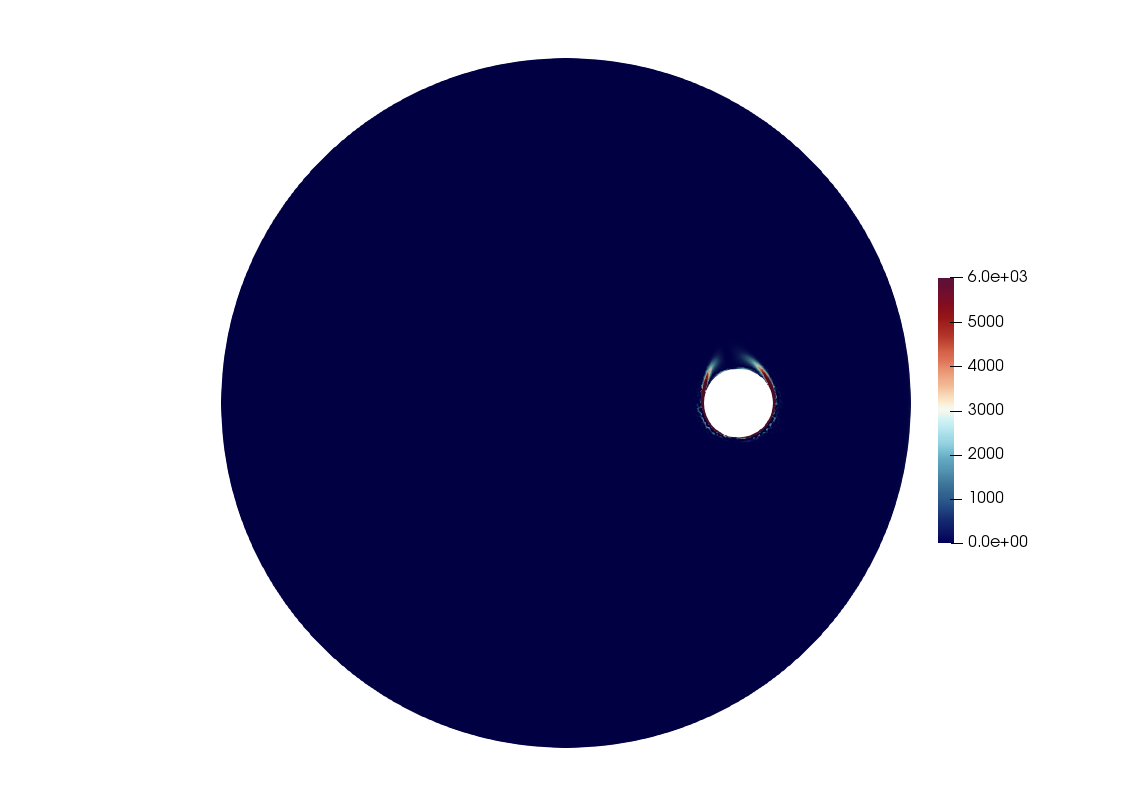}\label{fig31:b}}\newline
\subfloat[Velocity
$(t=5)$]{\includegraphics[width=.49\textwidth]{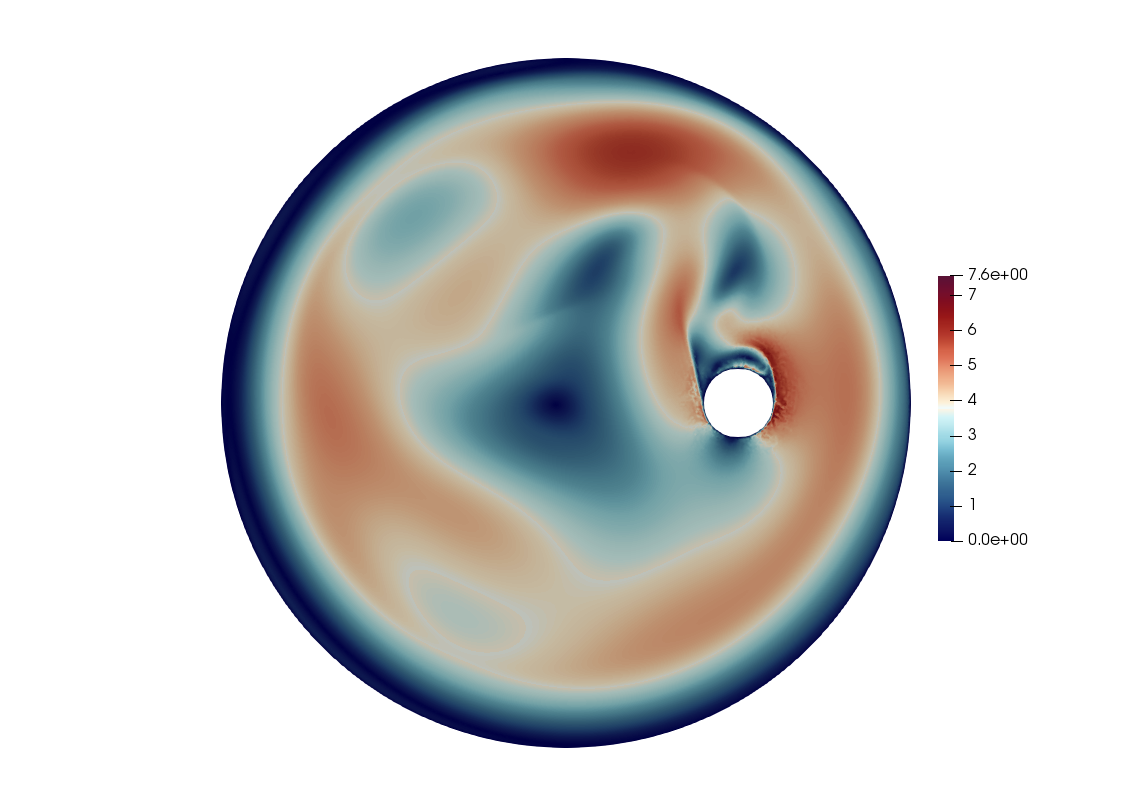}\label{fig31:c}}
\subfloat[Squared vorticity $(t=5)$]{\includegraphics[width=.49\textwidth]{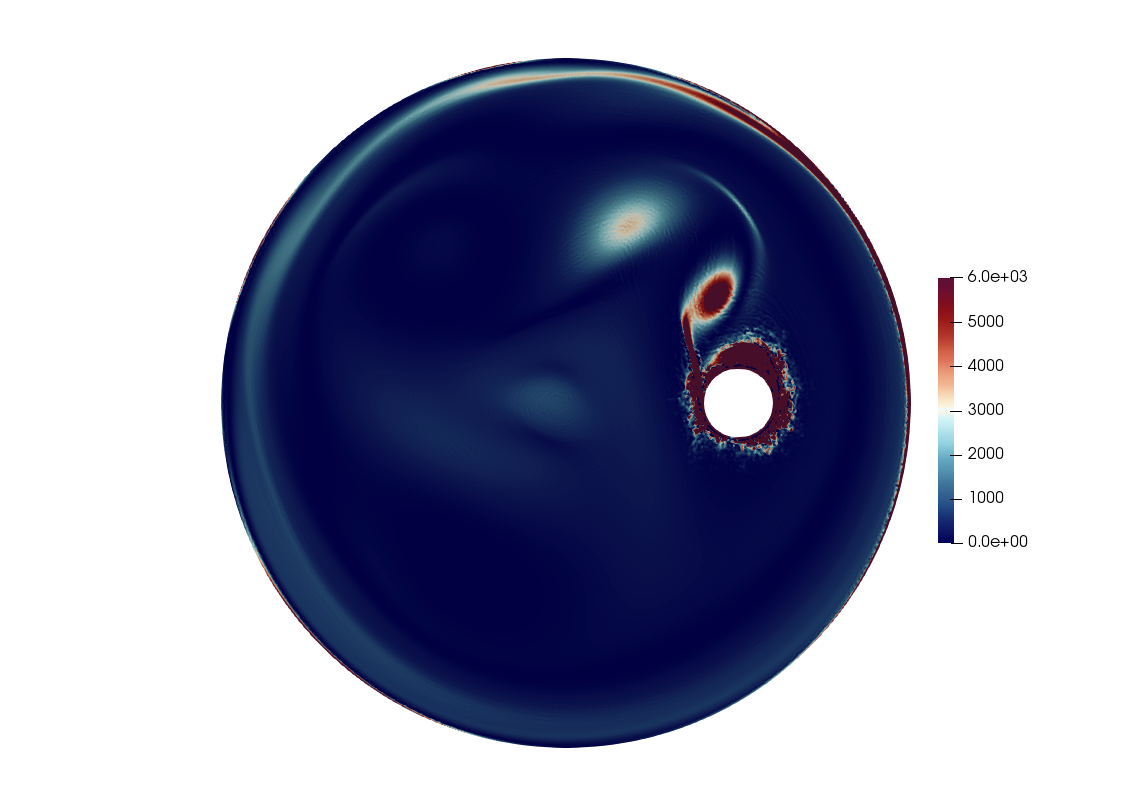}\label{fig31:d}}
\phantomcaption
\label{fig=flow1}\end{figure}\begin{figure}[t]
\ContinuedFloat
\subfloat[Velocity $(t=10)$]{\includegraphics[width=.49\textwidth]{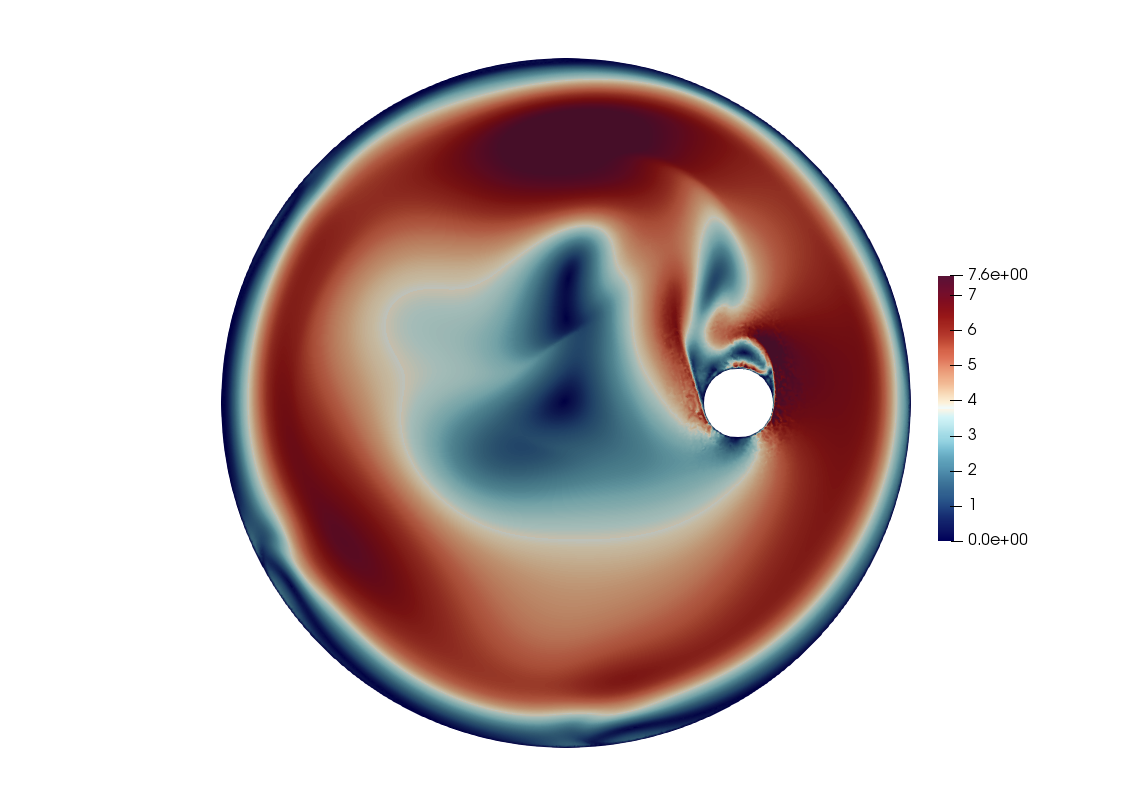}\label{fig31:e}}
\subfloat[Squared vorticity $(t=10)$]{\includegraphics[width=.49\textwidth]{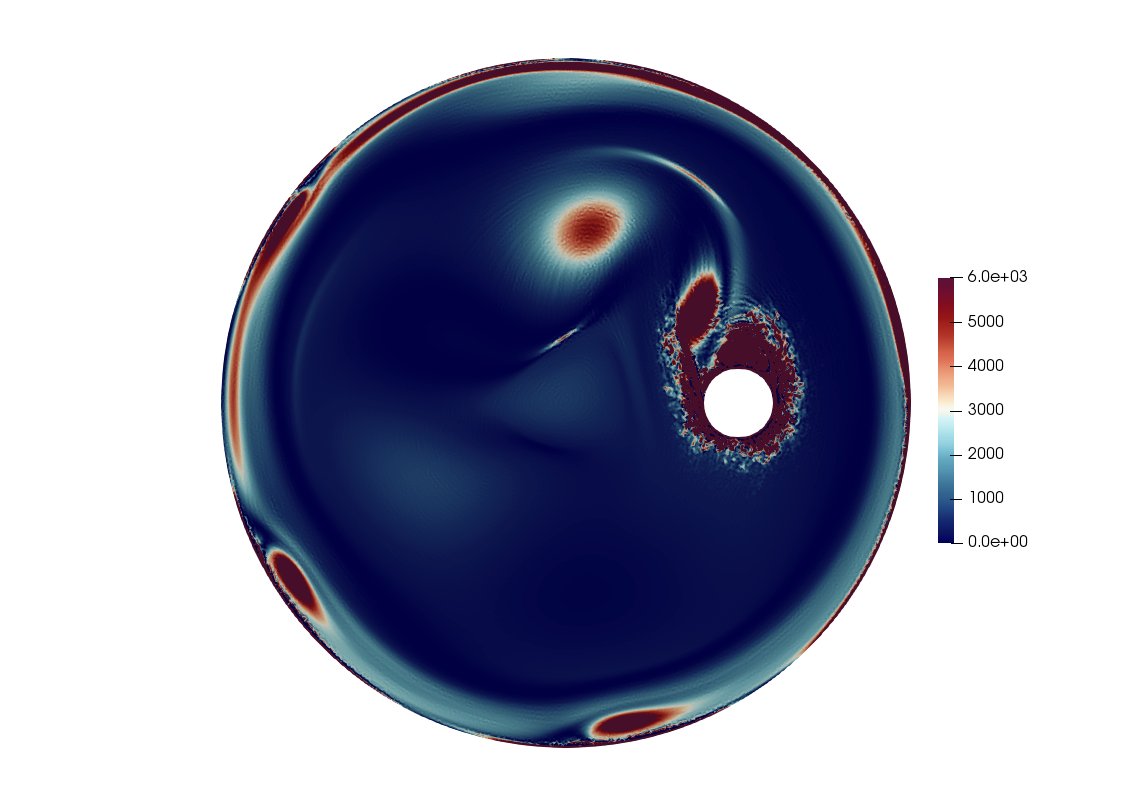}\label{fig31:f}}
\caption{Kinematic mixing length model velocity and vorticity.}
\label{fig=flow2}
\end{figure}

\subsection{Test 2: Flow between 3d offset cylinders}

The second test is a 3d analogue of the first. It shows similar differences in
the two models. Taking $\Omega_{1}$ to be the domain given in the first test,
we define $\Omega=\Omega_{1}\times(0,1)$, a cylinder of radius and height one
with a cylindrical obstacle removed. The domain $\Omega$ was discretized with
Delaunay tetrahedrons with a maximal mesh width of approximately $0.1$. As
before, we start the flow from rest $(v_{0}=(0,0,0)^{T})$ and let the
kinematic viscosity $\nu=0.0001$. The flow evolves via the body force
\[
f(x,y,z;t)=\min\{t,1\}(-4y(1-x^{2}-y^{2}),4x(1-x^{2}-y^{2}),0)^{T},
\]
and is observed over the time interval $(0,10]$, with $\Delta{t}=.05$ and the
initial conditions for $k$ being set in the same way as the first test. Below,
we present the evolution of the statistics introduced above.
\begin{figure}[t]
\subfloat[Model intensity
$I_{\text{model}}$]{\includegraphics[width=.49\textwidth]{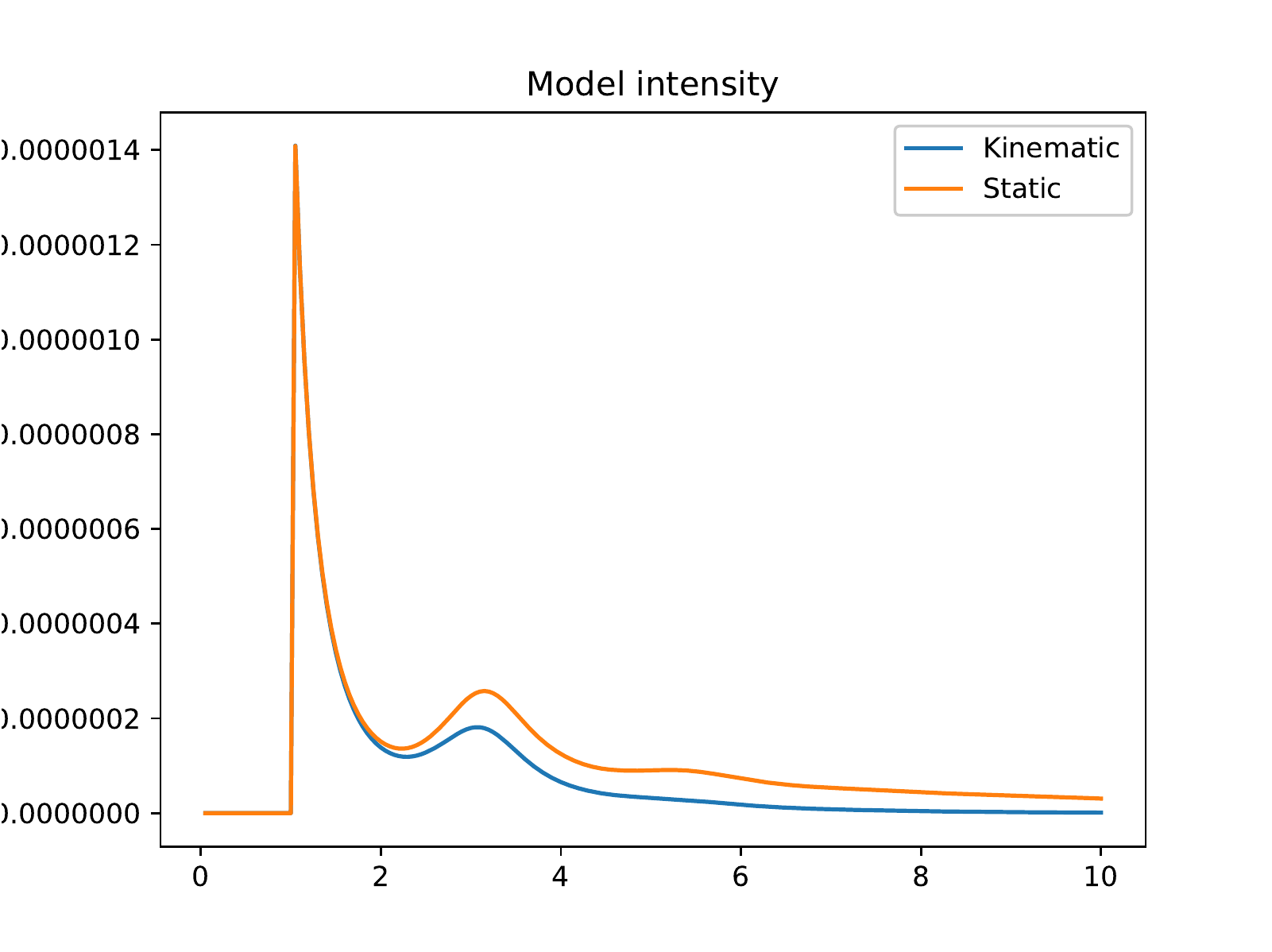}\label{fig3:a}}
\subfloat[Effective viscosity
$\nu_{\text{effective}}$]{\includegraphics[width=.49\textwidth]{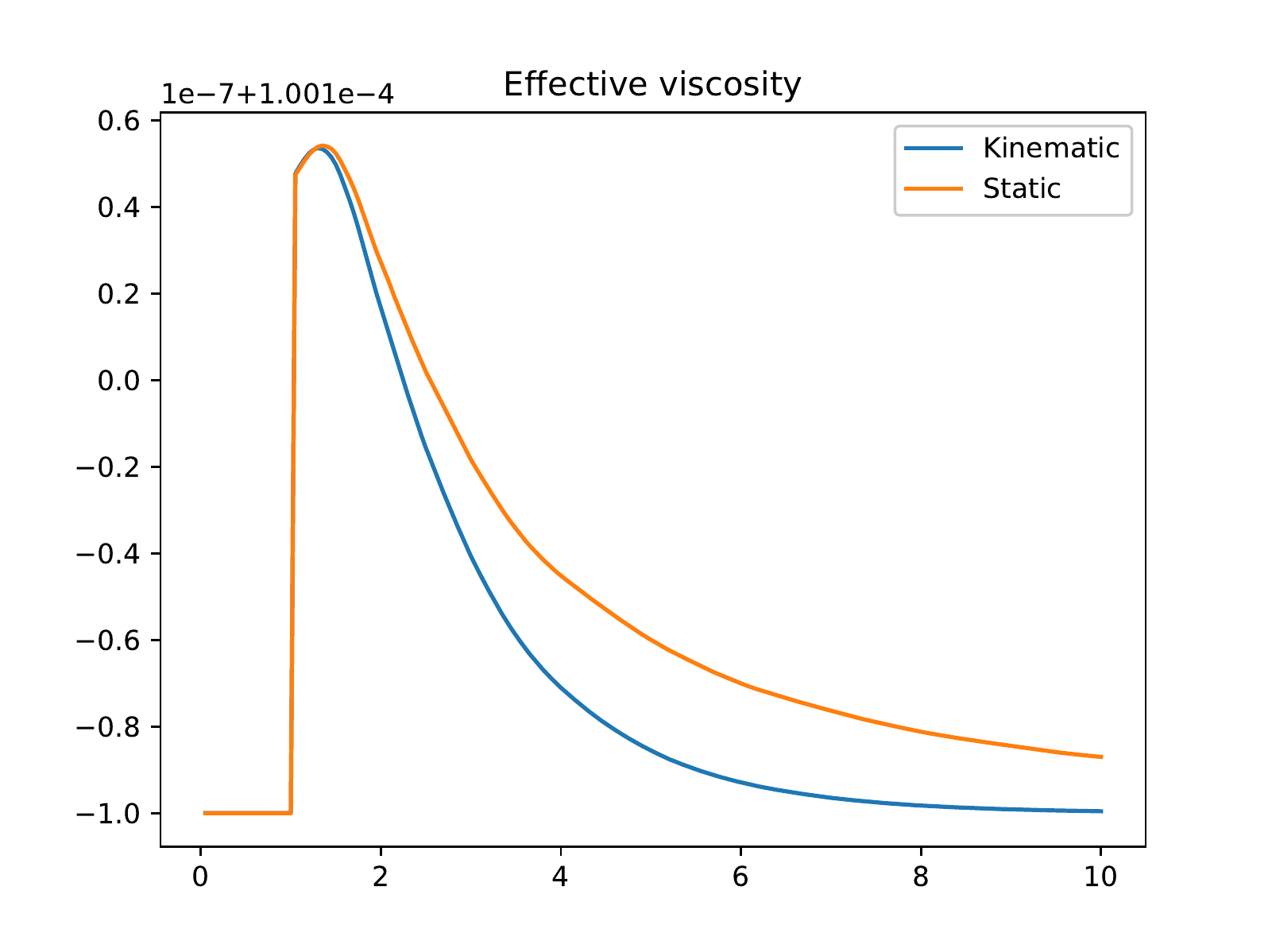}\label{fig3:b}}\newline%
\subfloat[Viscosity ratio
$VR_1$]{\includegraphics[width=.49\textwidth]{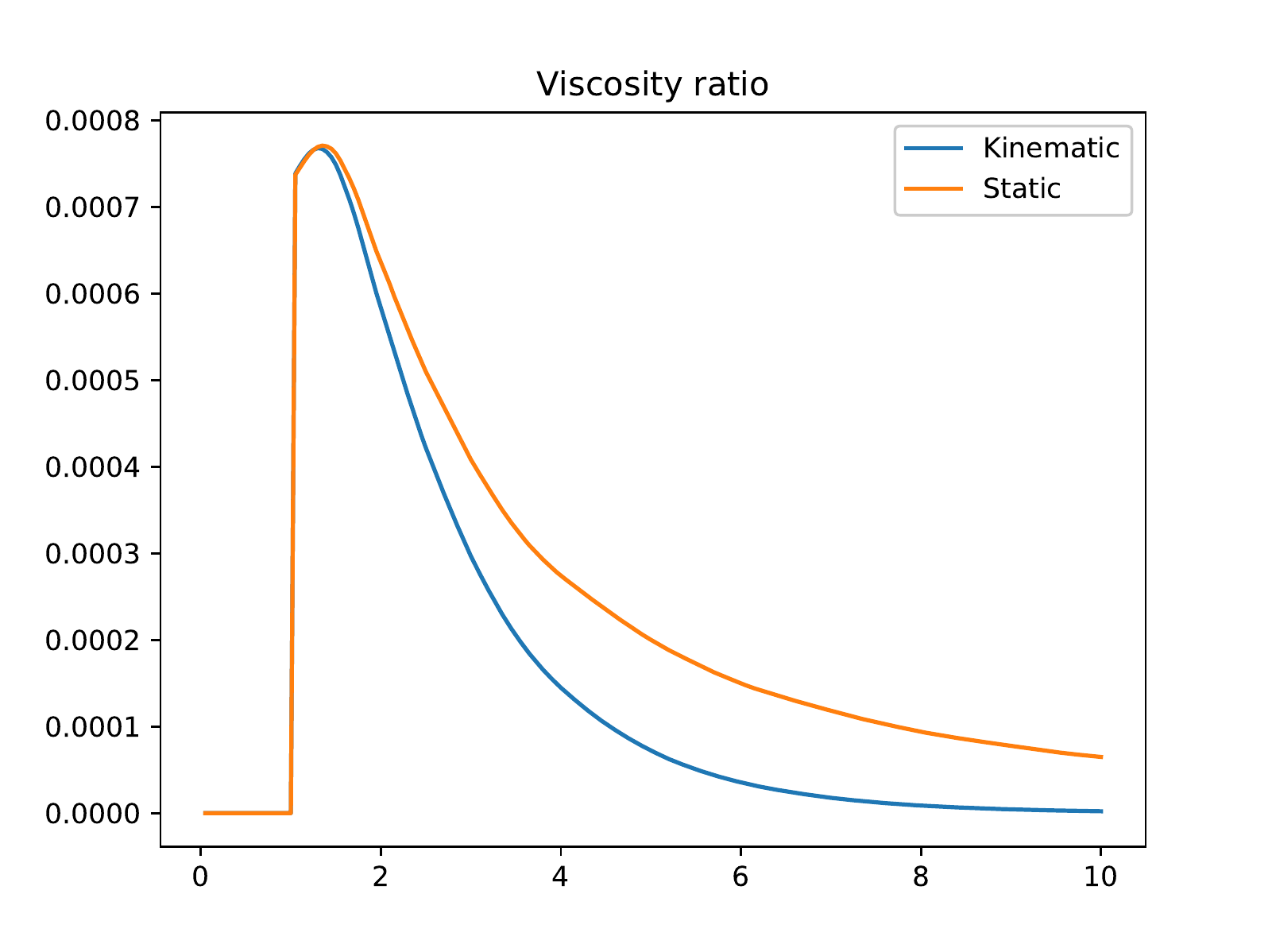}\label{fig3:c}}
\subfloat[Taylor microscale
$\lambda_{\text{Taylor}}$]{\includegraphics[width=.49\textwidth]{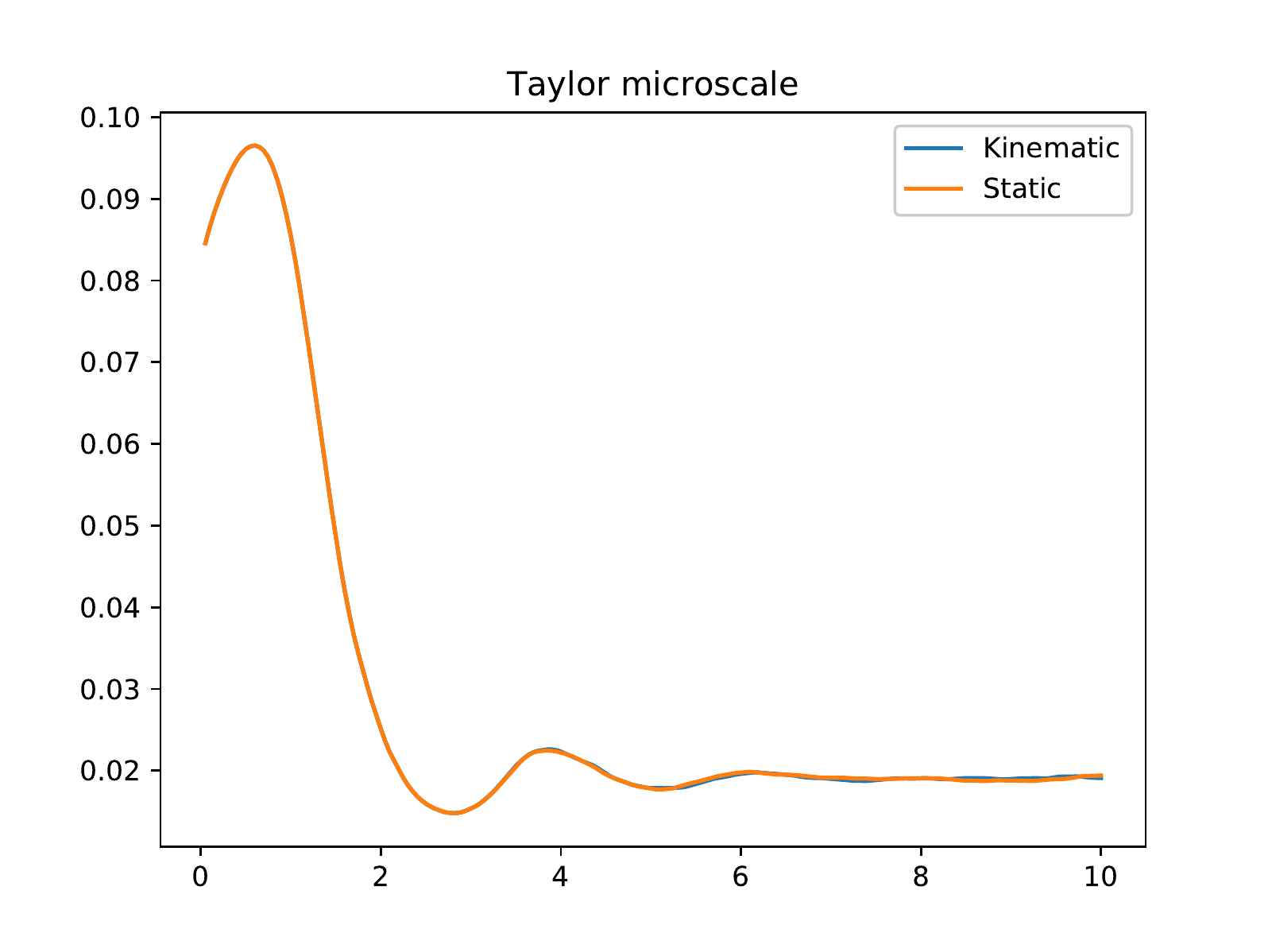}\label{fig3:d}}\newline
\subfloat[$avg{(l)}/L$]{\includegraphics[width=.49\textwidth]{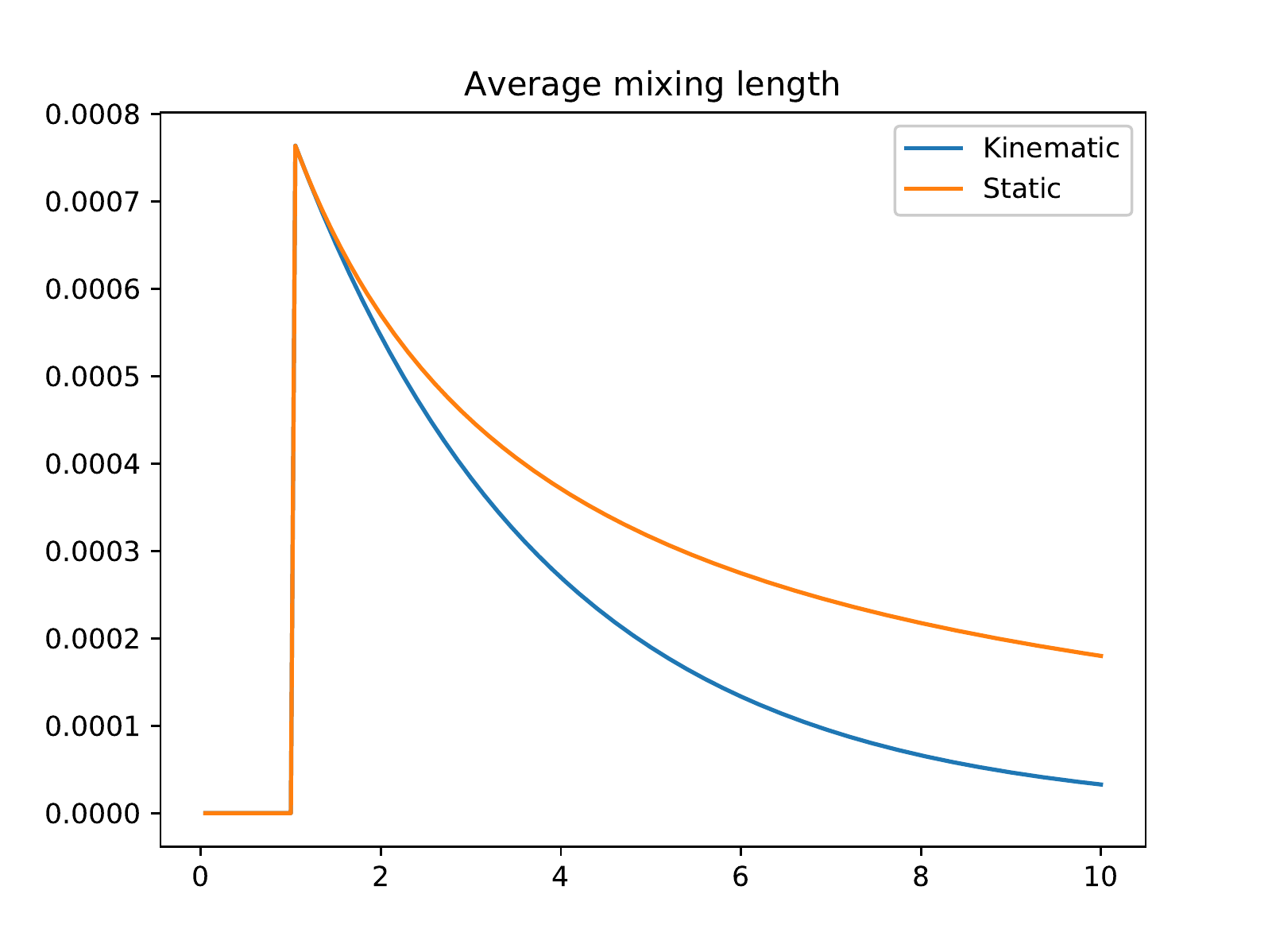}\label{fig3:e}}
\subfloat[$avg{(\nu_T)}/UL$]{\includegraphics[width=.49\textwidth]{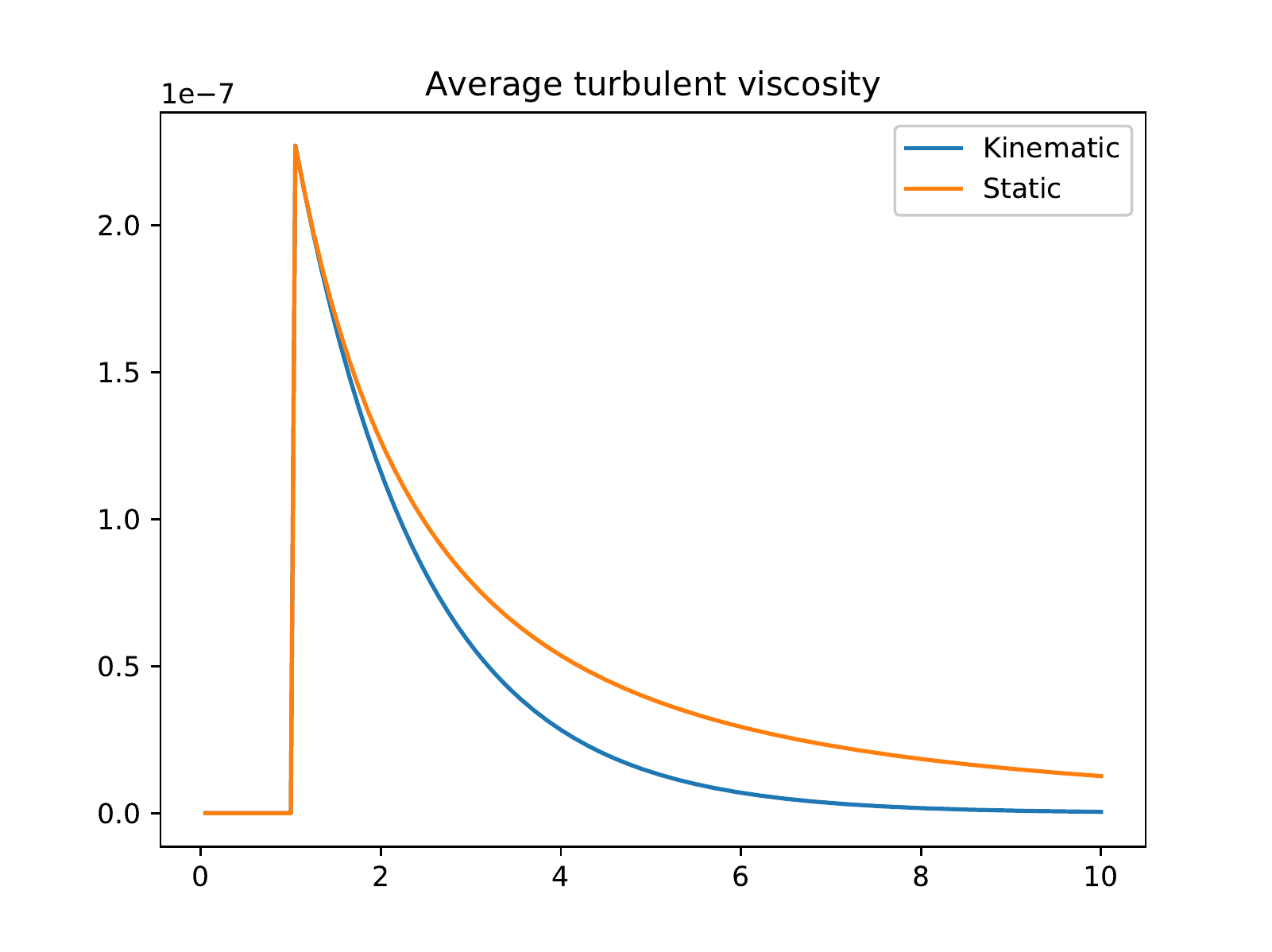}\label{fig3:f}}\caption{Flow
statistics for the 3d offset cylinder problem.}
\label{fig=stats_2}
\end{figure}

The statistics shown in Figure \ref{fig=stats_2} exhibit similar differences
between the 2 models as in the 2d case, Figs. \ref{fig3:a}--\ref{fig3:c},
\ref{fig3:e}--\ref{fig3:f}. As before, the evolution of the Taylor microscale
in Figure \ref{fig3:d} is similar in both models, with slight differences
appearing as the flow evolves. Here the Taylor microscale is much smaller for
the 3d test than the previous 2d test (even though the mesh is coarser). 

To conclude, we present streamline plots of the offset cylinder simulation as viewed from above. In the figures, color signifies the magnitude of velocity. At $t=1$, the flow appears laminar, and over the course of the simulation becomes turbulent, as evidenced by the plots at $t=5,10$. This behavior can be seen in Figure \ref{fig=3dcyl}, which views the domain from the positive $y$ direction and considers a slice at $z=.1$.
\begin{figure}[t]
\subfloat[$t=1$]{\includegraphics[width=.49\textwidth]{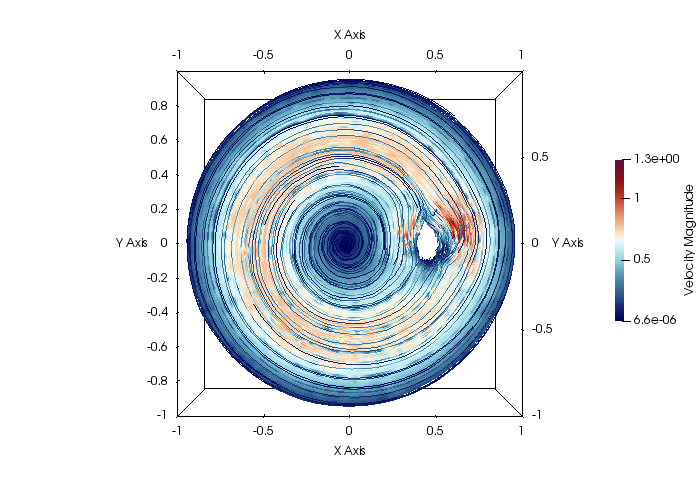}\label{fig4:a}}
\subfloat[$t=5$]{\includegraphics[width=.49\textwidth]{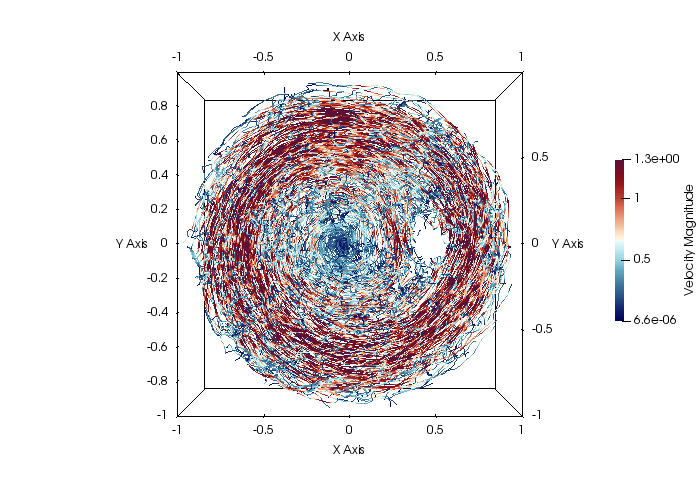}\label{fig4:b}}\newline
\subfloat[$t=10$]{\includegraphics[width=\textwidth]{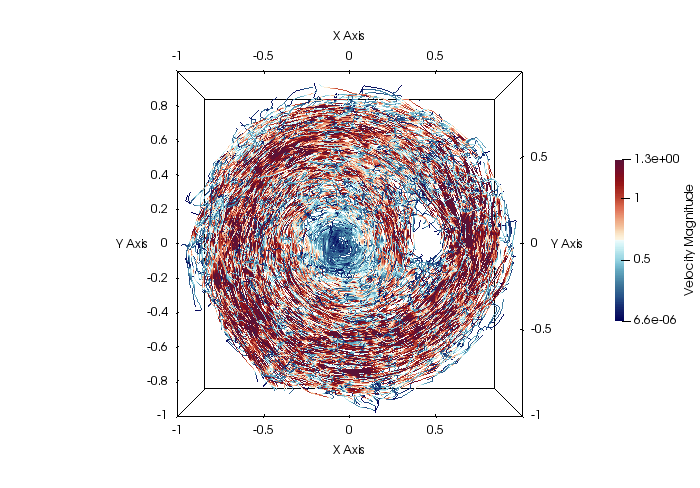}\label{fig4:c}}
\caption{Streamlines for the 3d offset cylinder problem.}
\label{fig=3dvel}
\end{figure}

\begin{figure}[t]
\subfloat[$t=1$]{\includegraphics[width=\textwidth]{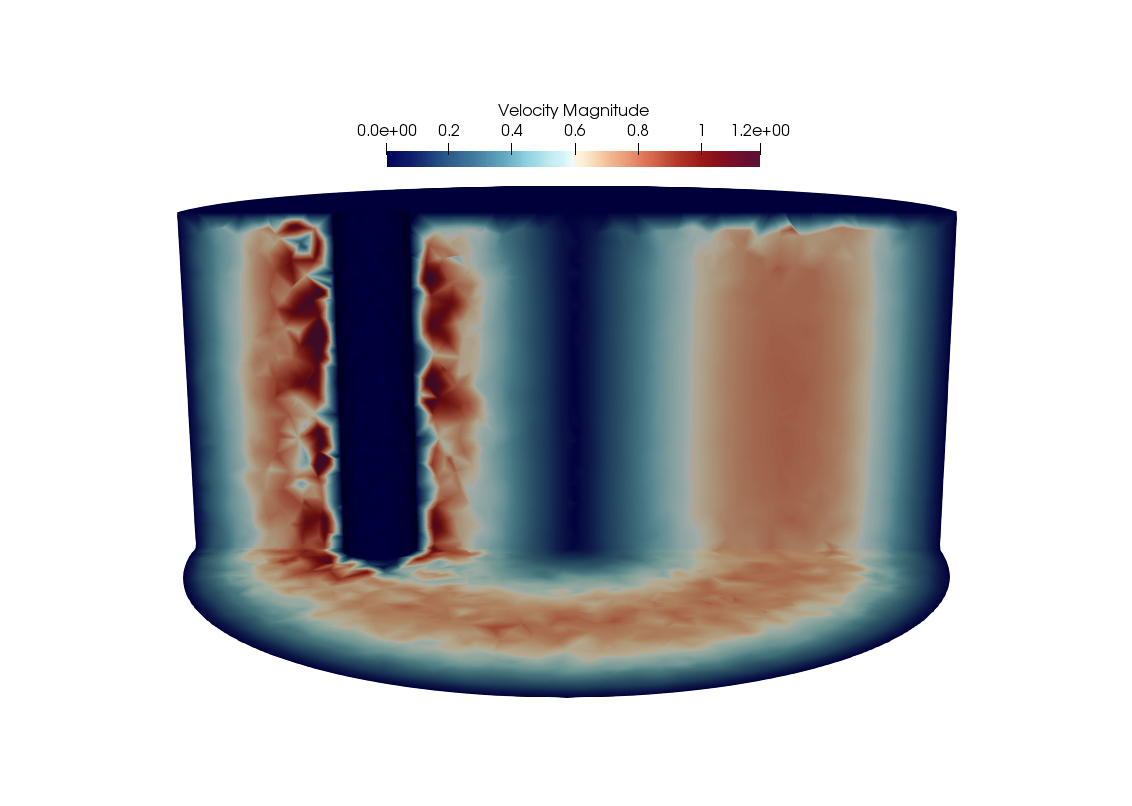}\label{fig5:a}}\newline
\subfloat[$t=5$]{\includegraphics[width=.49\textwidth]{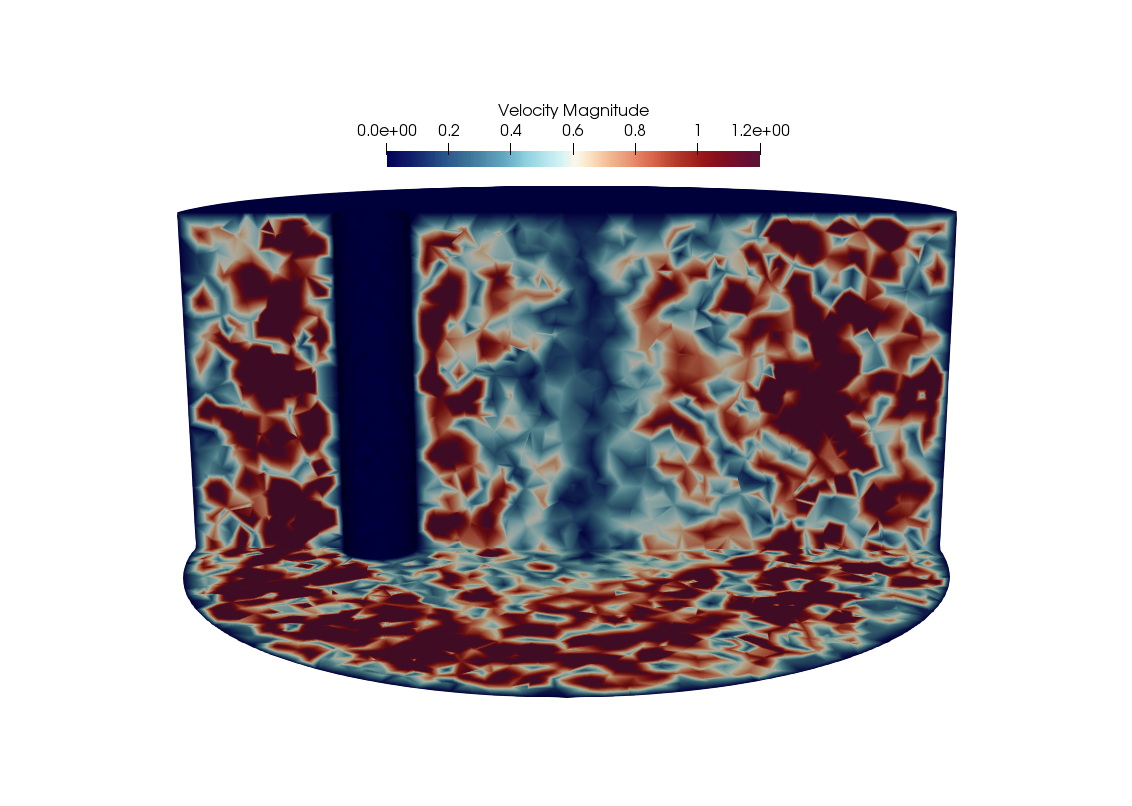}\label{fig5:b}}
\subfloat[$t=10$]{\includegraphics[width=.49\textwidth]{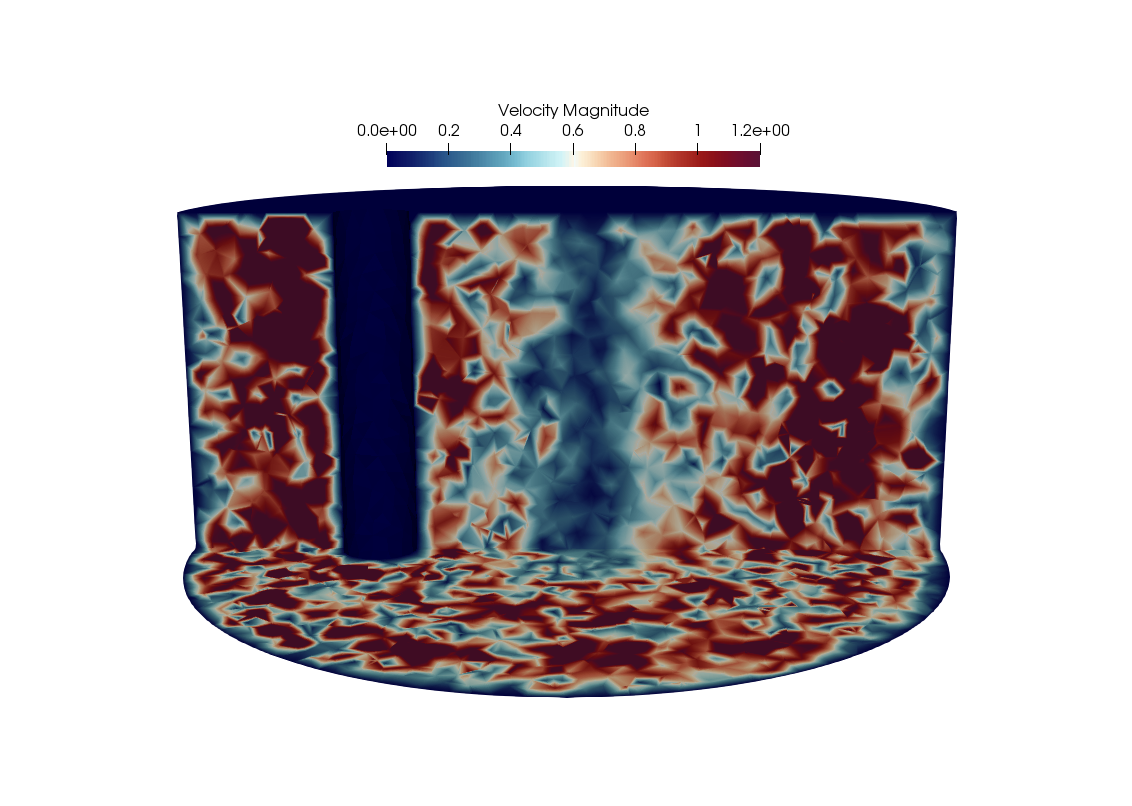}\label{fig5:c}}
\caption{Velocity magnitude for the 3d offset cylinder problem.}
\label{fig=3dcyl}
\end{figure}

\section{Conclusions and open problems}

Predictive simulation of turbulent flows using a URANS model requires some
prior knowledge of the flow to calibrate the model and side conditions. Our
intuition is that \textit{the better the model represents flow physics the
less complex this calibration will be}. To this end we have suggested a simple
modification of the standard $1-$equation model that analysis shows better
represents flow physics.

In turbulence, it is of course easier to list open problems than known facts.
However, there are a few within current technique for the modified model herein.

\begin{itemize}
\item Extension of estimates of $\left\langle \varepsilon_{\text{model}%
}\right\rangle $\ to turbulent shear flows is open and would give insight into
near wall behavior. Various methods for reducing the turbulent viscosity
locally in regions of persistent, coherent structures have been proposed,
e.g., \cite{V04}, \cite{LRT12} . Sharpening the (global) analysis of
$\left\langle \varepsilon_{\text{model}}\right\rangle $ for these (local)
schemes would be a significant breakthrough.

\item Extension of an existence theory to the modified model is another
important open problem. Our intuition is that existence will hold but there
may always occur hidden difficulties.

\item The estimate in Theorem 1 requires an upper limit on the time average's
window of $\tau/T^{\ast}\leq\mu^{-1/2}$. We do not know if a restriction of
this type can be removed through sharper analysis or if there exists a
fundamental barrier on the time average's window. Connected with this
question, the behavior of the model as $\tau\rightarrow\infty$\ is an open problem.

\item Eddy viscosity models do not permit transfer of energy from fluctuations
back to means. Recently in \cite{JL16} an idea for correcting these features
of eddy viscosity models was developed. Extension to the present context would
be a significant step forward in model accuracy.

\item Various averages of the classic turbulence length scale with the
kinematic one proposed herein are possible, such as the geometric average%
\[
l_{\theta}(x,t)=l_{0}^{\theta}(x)l_{K}^{1-\theta}(x,t).
\]
It is possible that such a weighted combination will perform better than
either alone. For example, for decaying turbulence when $v=0,\nabla v=0$ the
$k-$equation reduces to
\[
k_{t}+\frac{1}{l_{\theta}}k\sqrt{k}=0.
\]
Decaying turbulence experiments in 1966 of Compte-Bellot-Corsin, e.g., p.56-57
in \cite{MP}, suggest polynomial decay as $k(t)=$ $k(0)\left(  1+\lambda
t\right)  ^{-1.3}$. Neither mixing length formula replicates this decay. But
choosing $\theta=\frac{2}{1.3}\simeq1.54$ yields polynomial decay with
exponent $-1.3$. The effect of this data-fitting on the predictive power of
the model and on the Conditions 1-4 are an open problem.

\item Our intuition is that for many tests numerical dissipation is greater
than model dissipation (and acts on different features and scales of those
features). Thus the analysis of numerical dissipation including time
discretizations is an important open problems. 

\item Comparative test on problems known to be challenging for RANS and\\
URANS models is an important assessment step.
\end{itemize}

\end{document}